\documentclass[a4paper,11pt,reqno]{amsart}
\usepackage[latin1]{inputenc}

\usepackage{multicol}
\usepackage{amsfonts, amssymb, amsmath, amsthm, amscd,epsfig,mathrsfs}
\usepackage{latexsym}
\usepackage{graphicx}

\usepackage{cancel}
\usepackage[normalem]{ulem}

\usepackage[all]{xy}

\usepackage{xcolor}

\usepackage{enumerate}

\usepackage{verbatim}

\usepackage{hyperref}

\newtheorem{theorem}{Theorem}[section]

\newtheorem{corollary}[theorem]{Corollary}

\newtheorem{lemma}[theorem]{Lemma}

\newtheorem{proposition}[theorem]{Proposition}

	{\par\noindent{\bf Proposition \ref{res:hiper}.}\!\!
	\nopagebreak\normalsize\it}% 

	{\par\noindent{\bf Theorem \ref{result43}.}\!\!
	\nopagebreak\normalsize\it}% 

	{\par\noindent{\it Sketch of the proof}.  
	\nopagebreak\normalsize}% 
	{\hfill\linebreak[2]\hspace*{\fill}$\circlearrowleft$}

	{\par\noindent{\it Proof of Proposition }\ref{prop:stab:smc}.  
	\nopagebreak\normalsize}% 
	{\hfill\linebreak[2]\hspace*{\fill}$\circlearrowleft$}

	{\par\noindent{\it Proof of Propositions }\ref{adap:mon}{\it and }\ref{simult:adap}.\!\!\!
	\nopagebreak\normalsize}% 
	{\hfill\linebreak[2]\hspace*{\fill}$\circlearrowleft$}

{\theoremstyle{definition}
       \newtheorem{definition}[theorem]{Definition}

       \newtheorem{remark}[theorem]{Remark}
       \newtheorem{example}[theorem]{Example}
    
       \newtheorem{parrafo}[theorem]{{\!}}}

\numberwithin{equation}{theorem}

\DeclareMathOperator{\Max}{\underline{Max}}
\DeclareMathOperator{\mult}{mult}

\DeclareMathOperator{\Proj}{Proj}

\DeclareMathOperator{\Sing}{Sing}
\DeclareMathOperator{\Spec}{Spec}

\DeclareMathOperator{\In}{In}
\DeclareMathOperator{\Gr}{Gr}

\DeclareMathOperator{\Bl}{Bl}
\newcommand{\Diff}{\mathit{Diff}}
\DeclareMathOperator{\characteristic}{char}

\newcommand{\G}{{\mathcal G}}

\renewcommand{\H}{{\mathcal H}}

\newcommand{\q}{{\mathfrak q}}
\newcommand{\p}{{\mathfrak p}}

\newcommand{\TODO}[1]{\bgroup\color{red}\textbf{[}#1\textbf{]}\egroup}

\title{Finite morphisms and simultaneous reduction of the multiplicity}

\author{Carlos Abad}
\author{Ana Bravo}
\author{Orlando E. Villamayor U.}

\thanks{The authors were partially  supported from the Spanish Ministry of Economy and Competitiveness, through the ``Severo Ochoa'' Programme for Centres of Excellence in R\&D (SEV-2015-0554), and through MTM2015-68524-P (MINECO/FEDER)}

\address{Dpto. Matem\'aticas,  Universidad Aut\'onoma de Madrid and Instituto de Ciencias Matem\'aticas CSIC-UAM-UC3M-UCM, Ciudad Universitaria de Cantoblanco, 28049 Madrid, Spain}

\email[Carlos Abad]{carlos.abad@uam.es}
\email[Ana Bravo]{ana.bravo@uam.es}
\email[Orlando E. Villamayor U.]{villamayor@uam.es}


\begin{document}

\begin{abstract}
Let $X$ be a singular algebraic variety defined over a field $k$, with quotient field $K(X)$. Let $s \geq 2$ be the highest multiplicity of $X$ and $F_s(X)$ the set of points of multiplicity $s$. If $Y\subset F_s(X)$ is a regular center and $X\leftarrow X_1$ is the blow up at $Y$, then the  highest multiplicity of $X_1$ is less than or equal to $s$.  A sequence of blow ups at regular centers $Y_i \subset F_s(X_i)$,  say $X \leftarrow X_1 \leftarrow \dotsb \leftarrow X_n$, is said to be a \emph {simplification} of the multiplicity if the maximum multiplicity of $X_n$ is strictly lower than that of $X$, that is, if $F_s(X_n) $ is empty. In characteristic zero there is an algorithm which assigns to each $X$ a unique simplification of the multiplicity. However, the problem remains open when the characteristic is positive. 

In this paper  we will study finite dominant morphisms between singular varieties  $\beta: X'\to X$ of generic rank $r \geq 1$ (i.e., $[K(X'):K(X)]=r$). We will see that, when imposing suitable conditions on $\beta$, there is a strong link between the strata of maximum multiplicity of $X$ and $X'$, say $F_{s}(X)$ and $F_{rs}(X')$ respectively. In such case, we will say that the morphism is strongly transversal. When  $\beta: X'\to X$ is strongly transversal  one can obtain information about the simplification of the multiplicity of $X$ from that of $X'$ and vice versa. Finally, we will see that given a singular variety $X$ and   a finite field extension $L$ of $K(X)$ of rank $r \geq 1$, one can construct (at least locally, in \'etale topology) a strongly transversal morphism $\beta: X'\to X$, where $X'$ has quotient field $L$.
\end{abstract}

\maketitle

\setcounter{tocdepth}{1}
\tableofcontents

\section{Introduction}
Let $X$ be a variety of a perfect field $k$. Along these notes, all varieties will be assumed to be reduced and irreducible. The multiplicity along points of $X$ defines an upper semi-continuous function, say $\mult : X \to \mathbb{N}$, which assigns to each point $\xi \in X$ the multiplicity of the local ring $\mathcal{O}_{X,\xi}$ with respect to its maximal ideal. Under the previous hypotheses, the multiplicity function along points of $X$ has the following property: $X$ is regular at a point $\xi$ if and only if $\mult(\xi) = 1$. Thus, one can characterize the set of singular points of $X$ by
\[
	\Sing(X) = \{ \xi \in X \mid \mult(\xi) \geq 2 \}.
\]%

Assume that the scheme $X$ is singular, i.e., that $\max\mult(X) \geq 2$. Set $s = \max\mult(X)$ and let $F_s(X)$ denote the set of points of maximum multiplicity of $X$.   It can be  proved that the multiplicity does not increase when blowing up  along regular centers contained in  $F_s(X)$ (for a proof of this fact see \cite{Orbanz};  there it is mentioned   that the statement was first proved by Dade    in \cite{D}).   Namely, given a closed regular center $Y \subset F_s(X)$ and the blow up of $X$ along $Y$, say $\pi_1 : X_1 = \Bl_Y(X) \to X$, one has that
\[
	\mult(\xi_1) \leq \mult(\varphi(\xi_1))
\]%
for every $\xi_1 \in X_1$. Thus $\max\mult(X_1) \leq \max\mult(X)$. A blow up of this form will be called a \emph{$F_s$-permissible blow up}. Attending to Dade's result, it is natural ask whether there exists a sequence of $F_s$-permissible blow ups, say
\[ \xymatrix @R=0pt @C=30pt {
	X=X_0 & \ar[l]_-{\pi_1} X_1 &  \ar[l]_-{\pi_2}  \ldots &  \ar[l]_-{\pi_l} X_l,
} \]%
such that
\[
	\max\mult(X_0)
	= \dotsb = \max\mult(X_{l-1}) > \max\mult(X_l).
\]%
A sequence with this property, whenever it exists, will be called a \emph{simplification of $F_s(X)$}. Note that, if one could find a simplification of $F_s(X)$ for every $s \geq 2$ and for every $X$ with $\max\mult(X) = s$, then one could also find a resolution of singularities of $X$ by iterating this process (the resolution would be achieved when the maximum multiplicity of the scheme, after a finite number of blow ups, lowers to $1$). This approach differs from Hironaka's proof of resolution of singularities, as he uses the Hilbert-Samuel function instead of the multiplicity. 
We refer here to \cite{CJS} where the semicontinuity of this invariant along points of a singular scheme is studied in the first chapter (Theorem1.34),
and the behavior under blowups is carefully treated in the second chapter.

The existence of a simplification of $F_s(X)$ has already been proved for the case in which $X$ is a singular variety defined over a field $k$ of characteristic zero (cf. \cite{multiplicidad}). However, the problem remains open in the case of positive characteristic. 

 In this work we will study  finite morphisms between singular varieties,  $\beta : X' \to X$. We will show that, under suitable conditions, there is a link between the stratum of maximum multiplicity of $X$, say $F_s(X)$, and that of $X'$, say $F_{s'}(X')$. In this setting we present  a notion of {\em blow up  of a finite morphism} so that if $Y\subset F_{s'}(Y)$ is regular, then $\beta(Y)\subset F_s(X)$ is regular too, and there is a commtutative diagram of blow ups and finite morphisms: 
 $$\xymatrix@R=15pt{X' \ar[d]^{\beta} & \ar[l] X'_1 \ar[d]^{\beta_1} \\ 
 X  & \ar[l] X_1.  } $$
(notice that this parallels the situation in which $X$ is embedded in some scheme $V$ and   $Y\subset X$ is a center; then after the blow up at $Y$, there is again an embedding of the strict transform of $X$ in the transform of $V$, say $X_1\subset V_1$ and a commutative diagram of blow ups and immersions). 

Now the idea is that, given a singular variety $X$ with maximum multiplicity $s \geq 2$, one can construct a suitable variety $X'$ with maximum multiplicity $s' \geq 2$, endowed with a suitable finite morphism $\beta : X' \to X$, in such way that one can obtain information about the process of simplification of $F_s(X)$ from that of $F_{s'}(X')$. These ideas are developed in more detail in the following lines.

\medskip

\paragraph*{\bf Transversal morphisms}

Let $\beta : X' \to X$ be a finite and dominant morphism of singular varieties over a perfect field $k$. Set $s = \max\mult(X)$ and let $r = [K(X'):K(X)]$ denote the generic rank of the morphism $\beta$. Using Zariski's multiplicity  formula for finite projections (see Theorem \ref{MultForm}) one can prove that the maximum multiplicity on $X'$ is bounded by $rs$. Moreover, if $\max\mult(X') = rs$, one can show that the image of $F_{rs}(X')$ sits inside $F_s(X)$ and that $F_{rs}(X')$ is mapped homeomorphically to its image in $F_s(X)$ (see section~\ref{condicion_ast}). When the previous equality holds, i.e., when $\max\mult(X') = rs$, we will say that the morphism $\beta : X' \to X$ is \emph{transversal}.

In this work we will study conditions under which, given a transversal morphism $\beta : X' \to X$, the set $F_{rs}(X')$ is mapped surjectively onto $F_s(X)$, in such a way that $F_{rs}(X')$ and $F_s(X)$ are homeomorphic. In addition, we will require   this condition to be  preserved by sequences of $F_s$-permissible blow ups. Namely, we will require that any sequence of $F_s$-permissible blow ups over $X$, say
\[ \xymatrix @R=0pt @C=30pt {
	X=X_0 & \ar[l]_-{\pi_1} X_1 &  \ar[l]_-{\pi_2}  \ldots &  \ar[l]_-{\pi_l} X_l
} \]%
induces a sequence of $F_{rs}$-permissible blow ups over $X'$ and a commutative 
\[ \xymatrix@R=15pt{
 	X' \ar[d]^{\beta} &
 		X'_1 \ar[l] \ar[d]^{\beta_1} &
 		\dotsb \ar[l] &
 		X'_{l-1} \ar[l] \ar[d]^{\beta_{l-1}} &
 		X'_l \ar[l] \ar[d]^{\beta_l} \\
 	X &
 		X_1 \ar[l] &
 		\dotsb \ar[l] &
 		X_{l-1} \ar[l] &
 		X_l, \ar[l]
} \]%
with the following properties:
\begin{enumerate}[i\textup{)}]
\item $\beta_i : X'_i \to X_i$ is is transversal for $i \leq l-1$;
\item $F_{rs}(X'_i)$ is homeomorphic to $F_s(X_i)$ for $i \leq l-1$; and
\item $\max\mult(X'_l) < rs$ if and only if $\max\mult(X_i) < s$.
\end{enumerate}
When a transversal morphism $\beta : X' \to X$ satisfies these properties for every sequence of $F_s$-permissible blow ups over $X$, we shall say that $\beta$ is \emph{strongly transversal}. In section~\ref{elimination1} we will show that, if two varieties $X$ and $X'$ are linked by a strongly transversal morphism, say $\beta : X' \to X$, then the processes of simplification of $F_s(X)$ and $F_{rs}(X')$ are equivalent, in the sense that any simplification of $F_s(X)$ induces a simplification of $F_{rs}(X')$ and vice-versa.

\medskip

\paragraph*{\bf Characterization of strong transversality} Let $X$ be a singular variety defined over a perfect field $k$ with maximum multiplicity $s \geq 2$. There is a graded $\mathcal{O}_X$-algebra that one can naturally attach to the stratum of maximum multiplicity of $X$, which we shall denote by $\G_X$. In characteristic zero, this algebra encodes information about the process simplification of $F_s(X)$ (see Remark~\ref{equivalent_resolutions}). Moreover, given a transversal morphism $\beta : X' \to X$, there is a natural inclusion of $\G_X$ into $\G_{X'}$ (see \cite[Proposition~6.3]{Carlos}). In section~\ref{demo_theo} we will show that strong transversality can be characterized in terms of the algebras $\G_X$ and $\G_{X'}$. Namely, in Theorem~\ref{thm:strong-homeo} we will prove that, if $k$ has characteristic zero, then $\beta : X' \to X$ is strongly transversal if and only if $\G_{X'}$ is integral over $\G_X$. In the case of positive characteristic, the characterization is just partial, as strong transversality implies that $\G_{X'}$ is integral over $\G_X$ but the converse fails in general. In Example~\ref{contraejemplo} we give a counterexample for the latter case.

\medskip

\paragraph*{\bf Construction of strongly transversal morphisms} The last section is devoted to the construction of strongly transversal morphisms. Consider a singular variety $X$ over a perfect field $k$ with maximum multiplicity $s \geq 2$ and field of fractions $K$. Given a field extension $L \supset K$ of rank $r = [L:K]$, we would like to find a singular variety $X'$ with field of fractions $L$ endowed with a strongly transversal morphism $\beta : X' \to X$. Although we do not know whether it is possible to find such $X'$ with a strongly transversal morphism $\beta : X' \to X$ in general, at least we can prove that it possible to do a similar construction locally in \'etale topology, as explained below.

Let $\xi \in F_s(X)$ be a fixed point in the stratum of maximum multiplicity of $X$. For an \'etale neighborhood of $X$ at $\xi$, say $\widetilde{X} \to X$, let $\widetilde{K}$ denote the field of fractions of $\widetilde{X}$, and set $\widetilde{L} = L \otimes_K \widetilde{K}$ (as it will be explained, when $X$ normal, $\widetilde{X}$ can be assumed to be irreducible). Note that, in this setting, $\widetilde{L}$ might not be a field. However, as by base change it is \'etale over $L$, it should be reduced and, as it is finite (of rank $r$) over $\widetilde{K}$, it should be a direct sum of fields over $\widetilde{K}$. In Theorem~\ref{rango_r}  we will show that, for a suitable \'etale neighborhood $\widetilde{X}$ of $X$ at $\xi$, it is possible to find a scheme $\widetilde{Z}$ with quotient ring $\widetilde{L}$ endowed with a strongly transversal morphism $\widetilde{\beta} : \widetilde{Z} \to \widetilde{X}$.

To conclude, we mention that,  for technical reasons, to prove this last result we will assume that $X$ is normal. Actually, the construction can also be done without this assumption: in the general case, however,   some details  are   rather technical  while the main ideas of the construction  remain the same. 
 
\medskip

\paragraph*{\bf On the organization of the paper}  The paper is organized in three parts. 
 In Part \ref{parte_1} we study transversal morphisms.  Section \ref{condicion_ast} is devoted to  presenting the notion of transversal morphisms;  properties of such morphisms are studied in  Proposition \ref{homeomorphism} and Corollary \ref{condicion_estrella_B}. The main purpose of  section \ref{sec:mult-and-int-closure}  is to give a local presentation of transversal morphisms. This is the content of Lemma \ref{lm:primera_pres}.  In Section \ref{Transv_transformations} we study the compatibility of transversalitiy with blow ups at regular equimultiple centers  (the {\em blow up of a finite morphism}). The main result is Theorem \ref{thm:Frs-sequence}, which leads us naturally to the notion of strongly transversal morphism (see Definition \ref{def:strongly-transversal}). 
 
 Part \ref{parte_2}  is entirely devoted to present the basic results on Rees algebras that we need to state and proof Theorems \ref{thm:strong-homeo} and \ref{rango_r}. Rees algebras will allow us to work with equations  with weights: this turns out to be a useful language to describe the maximum multiplicity locus of an algebraic variety defined over a perfect field $k$. 
 
To conclude, strongly transversal morphisms are studied in Part  \ref{parte_3}.  Here Theorems  \ref{thm:strong-homeo}  and \ref{rango_r} are proven.

We thank A. Benito, D. Sulca, and S. Encinas for stimulating discussions on this subject.

%%%%%%%%%%%%%%%%%%%%%%%%%%%%%%%%%%%%%%%%%%%%%%%%%
%%%%%%%%%%%
%%%%%%%%%%%
%%%%%%%%%%% LA CONDICION (*) Y PRIMEROS RESULTADOS
%%%%%%%%%%%
%%%%%%%%%%%
%%%%%%%%%%%%%%%%%%%%%%%%%%%%%%%%%%%%%%%%%%%%%%%%

\part{Transversality} \label{parte_1}

\section{Transversality and finite morphisms} 
\label{condicion_ast}

Let $(R,M)$ be a local Noetherian ring, and let $\mathfrak{a} \subset R$ be an $M$-primary ideal. Observe that, for $n > 0$, the quotient ring $R / \mathfrak{a}^n$ has finite length when regarded as an $R$-module. Let us denote this length by $\lambda(R / \mathfrak{a}^n)$. It can be shown that, for $n \gg 0$, the value of $\lambda(R / \mathfrak{a}^n)$ is given by a polynomial in $n$ with rational coefficients, say
$\lambda(R / \mathfrak{a}^n)
	= c_d \cdot n^d + c_{d-1} \cdot n^{d-1} + \dotsb + c_0
	\in \mathbb{Q}[n],
$
where $d = \dim(R)$, which is known as the Hilbert-Samuel polynomial of $R$ with respect to $\mathfrak{a}$. In addition, it can be shown that $c_d = \frac{e}{d!}$ for some $e \in \mathbb{N}$. The integer $e$ is called the \emph{multiplicity} of $R$ with respect to $\mathfrak{a}$, and we shall denote it by $e_{\mathfrak{a}}(R)$. The \emph{multiplicity} of a Noetherian scheme $X$ at a point $\xi \in X$ is defined as that of the local ring $\mathcal{O}_{X,\xi}$ with respect to its maximal ideal.

\begin{parrafo} \label{condiciones}  {\bf General setting and notation.} Let $X$ and $X'$ be algebraic varieties over a perfect field $k$ with quotient fields $K = K(X)$ and $L = K(X')$ respectively. Suppose that $\beta: X' \to X$ is a finite dominant morphism of \emph{generic rank} $r := [K(X'):K(X)]$. We will be assuming that the maximum multiplicity of $X$ is $s \geq 1$, and will denote by $F_s(X)$ the (closed) set of points of maximum multiplicity $s$ of $X$. As we will see, under these conditions, the maximum multiplicity at points of $X'$ is at most $rs$, and we denote by $F_{sr}(X')$ the closed set (possibly empty) of points of multiplicity $sr$ of $X'$. Since our arguments will be of local nature, in several parts of this paper  we may assume that both, $X$ and $X'$ are affine, say $X = \Spec(B)$ and $X' = \Spec(B')$,  and that there is a finite extension $B\subset B'$. In this case we will use $F_s(B)$ (resp. $F_{sr}(B')$) to denote $F_s(X)$ (resp. $F_{sr}(X')$).   Moreover, we will see that the previous discussion also applies to the case in which $X'$ is an equidimensional scheme of finite type over $k$ with total ring of fractions $K(X')$.
\end{parrafo}

\begin{theorem}[{Zariski's multiplicity  formula for finite projections  \cite[Theorem 24, p. 297]{ZS}}] \label{MultForm}  Let 
$(B,{\mathfrak m})$ be a local domain, and let $B'$ be a finite extension of $B$. Let $K$ denote the quotient field of $B$, and let $L = K \otimes_B B'$.  Let $M_1, \ldots, M_r$ denote the maximal ideals of the semi-local ring $B'$, and assume that $\dim B'_{M_i}=\dim B$, $ i=1, \dots ,r$. Then
$$e_B({\mathfrak m})[L:K] = \sum_{1\leq i \leq r} e_{B'_{M_i}}({\mathfrak m} B_{M_i}) [k_i:k],$$
where $k_i$ is the residue field of $ B'_{M_i}$, $k$ is the residue field of $(B,{\mathfrak m})$, and $[L:K]= \dim_K L$.

\end{theorem}

As a consequence of Zariski's formula, if   $P$ is a prime ideal in $B'$ and $\p=P\cap B$, one has that 
\begin{equation}
\label{cota_multi}
e_{B'_{P}}(P B'_{P})\leq   e_{B'_{P}}(\p B'_{P})  \leq e_{B_{\p}}(\p B_{\p})[L:K].
\end{equation} 
That is,  the maximum of the multiplicity at points of $\Spec(B')=X'$ is bounded by the generic rank of the projection times the maximum of the multiplicity at points of $\Spec(B)=X$. Namely,
\begin{equation}\label{eq222} \max\mult(X') \leq [L:K] \cdot \max\mult(X).
\end{equation}

\begin{parrafo} \label{condicion_estrella} {\bf The (*) condition.} With the same notation and hypotheses as in Theorem \ref{MultForm} 
we will say that {\em condition  \textup{(*)} holds at $P\in \text{Spec}(B')$} if: 
$$(*)  \  e_{B'_P}(PB'_P)= e_{B_{\p}}(\p B_{\p})[L:K].$$
Note that, in particular, $e_{B'_P}(PB'_P) \leq e_{B'_P}(\p B'_P)$.  Hence,  condition (*) is satisfied at $P$ if and only the following three conditions hold simultaneously:
\begin{enumerate}
\item[(i)] $P$ is the only prime in $B'$ dominating ${\p}$ (i.e., $B'_P=B'\otimes_BB_{\p}$);

\item[(ii)]  $B_{\p}/{\p}B_{\p}=B'_P/PB'_P$;

%\item[(iii)]  $e_{B'_P}(\p B'_P)=e_{B'_P}(PB'_P)$,  (i.e., $PB'_P$ is the integral closure of ${\p}B'_P$ in $B'_P$). 
%ESTO ES UN TEOREMA DE NORTHCOTT.

\item[(iii)] $e_{B'_P}(\p B'_P)=e_{B'_P}(PB'_P)$.
\end{enumerate}
In particular, condition (*) necessarily holds for all primes $P\subset B'$  where the multiplicity is $sr$, where $s$ is the maximum multiplicity in $\Spec(B)$,  and $r=[L:K]$. 
\end{parrafo}

\begin{remark} Suppose that  $B$ and $B'$ are formally equidimensional locally at any prime. Then, 
as we will see in Section~\ref{sec:mult-and-int-closure}, condition (iii) is equivalent to saying that $\p B'_P$ is a reduction of $P B'_P$, i.e., that the ideal $P B'_P$ is integral over $\p B'_P$ (cf. Theorem~\ref{Rees}).
\end{remark}

\begin{definition} \label{def:transversal}
Let $X$ be a variety over a perfect field $k$, and let $X'$ be either a variety or a reduced equidimensional scheme of finite type over $k$. Let $K$ be the field of rational functions of $X$ and let $L$ be the total ring of fractions of $X'$. We will say that a $k$-morphism $\beta : X' \to X$ is \emph{transversal} if it is finite and dominant, and the following equality holds:
\[
\max\mult(X')= [L:K] \cdot \max\mult(X).
\]%
We will say that an extension of $k$-algebras of finite type $B \subset B'$, where $B$ is a domain, and $B'$ is reduced and equidimensional, is \emph{transversal} if the corresponding morphism of schemes, say $\beta : \Spec(B') \to \Spec(B)$, is so.\color{black}	
	
\end{definition}

\begin{remark}
By definition, if $\beta : X' \to X$ is a transversal morphism of generic rank $r = [L:K]$ and $X$ has maximum multiplicity $s \geq 1$, then $X'$ has maximum multiplicity $rs$ and the equality is attained in \eqref{cota_multi} for all points in $F_{rs}(X')$. Thus condition (*) holds at all points of $F_{rs}(X')$.
\end{remark}

\begin{proposition} \label{homeomorphism} 
	Let $\beta: X'\to X$ be a transversal morphism of generic rank $r$ and  let $s=\max\mult(X)$. Then: 
\begin{enumerate} 
\item $\beta(F_{sr}(X'))\subset F_s(X)$;
\item $F_{sr}(X')$  is homeomorphic to $\beta(F_{sr}(X'))\subset  F_s(X)$. 
\end{enumerate} 
\end{proposition}

\begin{proof}
	 It suffices to argue locally, so we may assume that $X=\text{Spec} (B)$ and  $X'=\text{Spec} (B')$.   Now  it follows from the discussion in \ref{condicion_estrella} that  the maximum multiplicity at points of $\text{Spec} (B')$ is $rs$, and that 
  condition (*) necessarily holds for all primes  in $ F_{sr}(B')$. Furthermore,   if $Q\in F_{sr}(B')$ 
and $\q=Q\cap B$ then  necessarily $e_{B_{\q}}(\q B _{\q})=s$, so
$\beta(F_{sr}(X'))\subset  F_s(X)$. Therefore $\beta$ induces an injective morphism, say 
$\beta: F_{sr}(X') \to F_s(X)$, which is proper and hence it induces an homeomorphism $\beta: F_{sr}(X') \to \beta(F_{sr}(X))(\subset F_s(X))$.
\end{proof}

\begin{corollary} \label{condicion_estrella_B}  Suppose that the extension 
$B\subset B'$ is transversal of generic rank $r$.  Let $s$ be the maximum multiplicity at points of $\Spec(B)$, let  $Q\subset F_{sr}(B')$ and let $\q:=B\cap Q$. 
Then: 
\begin{enumerate}
\item[(1)] The extension $B/\q\to B'/Q$ is finite;
\item[(2)] The field of fractions of $B/\q$, say $K(\q)$, equals the field of fractions of $B'/Q$, say $K(Q)$; 
\item[(3)] If  $P\subset B'$ is a prime ideal containing $Q$,   $\p:=P\cap B$,  $\overline{P}=P/Q$,  and 
$\overline{\p}=\p/\q$, then 
$$\text{e}_{B/\q}(\overline{\p})=\text{e}_{B/\q}(\overline{\p})[K(Q): K(\q)]=\text{e}_{B'/Q}(\overline{P}).$$
\end{enumerate}
In other words, the finite extension $B/\q\to B'/Q$ is  birational and bijective, and the multiplicity at corresponding points is the same. 
In particular,  $B'/Q$ is regular if and only if $B/\q$ is regular, and in that case necessarily  $B/\q=B'/Q$. 
\end{corollary}
 
\begin{proof}
The statement in (1) is clear and (2) follows from condition (ii) of \ref{condicion_estrella}. As for (3) note here that $\Spec(B'/Q) \subset F_{sr}(B')$, and the conditions (i), (ii) and (iii) of \ref{condicion_estrella}, which hold for all $P$ containing $Q$, are inherited by $\overline{P}=P/Q (\subset B'/Q)$.
\end{proof}

\begin{parrafo} {\bf Summarizing.} \label{resumen} With the same notation and conventions as    \ref{condiciones}, if $\beta: X'\to X$ is transversal, then from Corollaries \ref{homeomorphism} and \ref{condicion_estrella_B}  it follows that: 
\begin{enumerate}  
	\item $\beta (F_{sr}(X')) \subset F_{s}(X)$;
	\item $F_{sr}(X')$ is homeomorphic to $\beta (F_{sr}(X'))$; 
	\item If  $Y \subset F_{sr}(X')$ is an irreducible regular subscheme, then $\beta(Y)\subset F_{s}(X)$ is an irreducible regular subscheme;
	\item If $Z \subset F_{s}(X)$ is an irreducible closed regular subscheme, and if $\beta^{-1}(Z) \subset F_{sr}(X')$, then $\beta^{-1}(Z)_{\text{red}}$ is regular. 
\end{enumerate}
\color{black}
\end{parrafo}

%%%%%%%%%%%%%%%%%%%%%%%%%%%%%%%%%%%%%%%%
%%%%%%%%%%%%%%%%%%%%%%%%%%%%%%%%%%%%%%%%
%%%%%
%%%%% PRESENTATIONS OF TRANSVERSAL MORPHISMS
%%%%%
%%%%%%%%%%%%%%%%%%%%%%%%%%%%%%%%%%%%%%%%
%%%%%%%%%%%%%%%%%%%%%%%%%%%%%%%%%%%%%%%%

\section{Presentations of transversal morphisms}
\label{sec:mult-and-int-closure}

The notion of multiplicity is closely related to that of integral closure of ideals. Given ideals $I\subset J$ in a Noetherian ring $B$, there are several (equivalent) formulations for $J$ to be integral over $I$, or say for $I$ to be a reduction of $J$.
Northcott and Rees introduced the notion of reduction. Given ideals $I\subset J$ as above, $I$ is a reduction of $J$ if $IJ^n=J^{n+1}$ for some $n \geq 0$ or, equivalently, if the inclusion $B[IW]\subset B[JW]$ is a finite extension of subrings in $B[W]$. In this case, if $X \to \Spec(B)$ denotes the blow up at $I$ and $X' \to \Spec(B)$ is the blow up at $J$, there is a factorization $X' \to X$ induced by this finite extension, which is also a finite morphism.

The notion of reduction of an ideal $J$ in $B$ will appear naturally when studying the fibers of the blow up, which is a notion that relates to that of the integral closure of $J$.

\begin{parrafo}\label{finito}
A first connection of integral closure with the notion of multiplicity is given by the following Theorem of Rees. A    Noetherian  local ring $(A,M)$   is said to be formally equidimensional (quasi-unmixed in Nagata's terminology) if  
$\dim(\hat{A}/{\hat{\mathfrak p}})=\dim(\hat{A})$ at each minimal prime ideal $\hat{\mathfrak p}$  in the completion $\hat{A}$.
\end{parrafo}

\begin{theorem}[Rees Theorem, \cite{Rees}] \label{Rees}
Let $(A,M)$ be a formally equidimensional local ring, and let $I\subset J$ be two $M$-primary ideals.
%Then $I$ and $J$ have the same integral closure if and only if $e_A(I)=e_A(J)$.
Then $I$ is a reduction of $J$ if and only if $e_A(I)=e_A(J)$.
\end{theorem}

\begin{parrafo} \label{110}{\em A generalization of Rees Theorem.} Theorem \ref{Rees} can be generalized to the case of (non-necessarily) $M$-primary ideals, see Theorem \ref{TB} below. Before stating the theorem, we will review some of the terms that appear among the hypotheses.  
Let $I$ denote an ideal in a local ring $(A,M)$. Let 
$f: Z \to \Spec(A)$
be the blow up at $I$, and let
$f_0: Z_0\to \Spec(A/I)$
be the proper morphism induced by restriction.
Northcott and Rees defined the {\em analytic spread} as: 
$$l(I) = \dim(A/M\otimes_A\Gr_I(A))=\delta +1,$$ where $\delta$ is the dimension of the fiber of $f$ over the closed point of $\Spec(A)$, or equivalently, the dimension of the fiber of $f_0$ over the closed point.
Note that $l(I) = \dim(A)$ when $I$ is $M$-primary.

The height of $I$, say $ht(I)$, is $\min(\dim A_\q)$ as $\q$ runs through all primes containing $I$, and we claim that $l(I) \geq ht(I)$, with equality if and only if all fibers of $f_0$ have the same dimension (see \cite[\S2]{Lipman2}).
The inequality holds because the dimension of the fibers of $f_0: Z_0\to \Spec(A/I)$ is an upper semi-continuous function on primes of $A/I$ (this result is due to Chevalley, cf. \cite[Theorem~13.1.3]{EGAIV}). In addition, if $\q$ is minimal containing $I$, the dimension of the fiber over $\q$ is dim $A_{\q}$.

Finally, let $I\subset J$ is a reduction in a Noetherian ring $B$, and  let $X \to \Spec(B)$ and $X' \to \Spec(B)$ denote the blow ups at $I$ and $J$  (respectively). Since there is a factorization $X' \to X$ which is finite, $l(IB_P)=l(JB_P)$ at any prime $P$ in $B$.
\end{parrafo}

\begin{theorem}[{B\"{o}ger, cf. \cite[Theorems 2 and 3]{Lipman2}}] \label{TB}
Let $(A,M)$ be a formally equidimensional local ring. Fix an ideal $I \subset A$ so that $ht(I) = l(I)$. Consider another ideal $J \subset A$ so that $I \subset J \subset \sqrt{I}$. Then $I$ is a reduction of $J$ if and only if $e_{A_{\q}} (IA_{\q}) = e_{A_{\q}} (JA_{\q})$
%\[
%	e_{A_{\q}} (IA_{\q})
%	= e_{A_{\q}} (JA_{\q})
%\]%
for each minimal prime $\q$ of $I$.
\end{theorem}

Along these notes we will consider and study the blow ups along regular centers. The following theorem relates the notion of equimultiplicity with that of the fiber dimension at the blow up (analytic spread), studied by Dade, Hironaka, and Schickhoff.

\begin{theorem}[{Hironaka-Schickhoff, \cite[Corollary~3, p.~121]{Lipman2}}] \label{HS}
Let $(A,M)$ be a formally equidimensional local ring, and let $\p \subset A$ be a prime ideal so that $A / \p$ is regular. Then $ht(\p)=l(\p)$ in $A$ if and only if the local rings $A$ and $A_{\p}$ have the same multiplicity.
\end{theorem}

\begin{lemma}[Presentation of transversal extensions]
\label{lm:primera_pres}
Let $B$ be a domain of finite type over a perfect field $k$ with maximum multiplicity $s$ \textup{(}$s\geq 1$\textup{)}, and let $B \subset B'$ be a transversal extension of generic rank $r$ \textup{(}see \ref{condiciones} and Definition~\ref{def:transversal}\textup{)}. Let $P \in F_{sr}(B')$ be a prime, and let $\p = B \cap P$. Then:
\begin{enumerate}
\item The extension of local rings $B_{\p} \to B'_P$ is finite. Moreover, there are elements $\theta_1, \dotsc, \theta_m \in B'_P$ integral over $\p B'_P$ such that $B'_P = B_{\p}[\theta_1, \dotsc, \theta_m]$;
\item If in addition $B'/P$ is a regular ring, then the $\theta_1, \dotsc, \theta_m$ can be chosen in $P$, and $P$ is the integral closure of $\p B'$ in $B'$.
\end{enumerate}%
\end{lemma}

\begin{proof}
(1) The extension $B_{\p}\to B'_P$ is finite, thus there are elements $\omega_1,\dotsc, \omega_m \in B'_P$ so that $B'_P = B_{\p}[\omega_1, \ldots, \omega_m]$. The quotient field of $B_{\p}$, say $K(\p)$, is equal to the quotient field of $B'_P$, say $K(P)$. Hence there are elements $\alpha_1, \ldots, \alpha_m\in B_{\q}$ such that $\overline{\alpha_i} = \overline{\omega_i} \in K(\p) = K(P)$. Setting $\theta_i:=\omega_i-\alpha_i$ one has that $\theta_i \in P$ and  $B'_P = B_{\q}[\theta_1, \dotsc, \theta_m]$. Since $PB'_P$ is integral over $\p B'_P$ each $\theta_i$ is integral over $\p B'_P$.  
 
(2) Set $B' = B[\theta_1, \dotsc, \theta_m]$. By condition~(3) of Corollary~\ref{condicion_estrella_B}, $B'/P=B/\p$. So clearly $\theta_1, \dots, \theta_m$ can be modified by taking $\theta_1-\alpha_1, \dots, \theta_m- \alpha_m$, $\alpha_i\in B$, and we may assume that they are in $P$. 

The discussion in \ref{condicion_estrella}, more precisely in condition (iii),  applies for the local rings $B_{\p}\subset B'_P$, and therefore $\p B'_P$ is a reduction of $PB'_P$ or say that $PB'_P$ is the integral closure of $\p B'_P$. Our claim now is that $P$ is the integral closure of $\p B'$ at the ring $B'$. It suffices to check that this claim holds locally at every maximal ideal $M\subset B'$ containing $P$. We will apply Theorem \ref{TB} at $B'_M$ to prove this latter claim:  
Let $\mathfrak{m}=M\cap B$, as $B'/P$ is regular and equal to $B/\p B$, it follows that $B_{\mathfrak{m}}/\p B_{\mathfrak{m}}$ is a regular local ring.
Under these conditions, as $B$ is a domain and the multiplicity of $B_{\mathfrak{m}}$ coincides with that of $B_\mathfrak{p}$, Theorem~\ref{HS} asserts that $l(\p)=ht(\p)$ in $B_{\mathfrak{m}}$.
 
Now, since the extension $B_{\mathfrak{m}}\to B'_M$ is finite, we have that $ht(\p B'_M) = ht(\p)$, and the blow up of $B'_M$ at $\p B'_M$ is finite over the blow up of $B_{\mathfrak{m}}$ at ${\p}B_{\mathfrak{m}}$. This implies that the fibers over the closed points have the same dimension. Thus $l(\p B'_M) = l(\p)$, and therefore $l(\p B'_M) = ht(\p B'_M)$ in $B'_M$. Recall that $e_{B'_P}(PB'_P)=e_{B'_P}(\p B'_P)$ by condition (iii) of \ref{condicion_estrella} and Theorem~\ref{Rees}. Now, since $P B'_M$ is the only minimal prime of $\p B'_M$, the claim follows from Theorem~\ref{TB} (taking $A = B'_M$, $I = \p B'_M$, and $J = P B'_M$).

\end{proof}

%%%%%%%%%%%%%%%%%%%%%%%%%%%%%%%%%%%%%%%%
%%%%%%%%%%%%%%%%%%%%%%%%%%%%%%%%%%%%%%%%
%%%%%
%%%%% TRANSFORMATIONS AND LOCAL SEQUENCES
%%%%%
%%%%%%%%%%%%%%%%%%%%%%%%%%%%%%%%%%%%%%%%
%%%%%%%%%%%%%%%%%%%%%%%%%%%%%%%%%%%%%%%%

\section{Transversality and blow ups}\label{Transv_transformations}

Let $X$ be an equidimensional scheme of finite type over a perfect field $k$ with maximum multiplicity $s$. As we will see, there are some types of transformations that play an important role in the study of the multiplicity.  We will say that a morphism $X \leftarrow X_1$ is an \emph{$F_s$-local transformation} if it is of one of the following types:
\begin{enumerate}
\item[i)] The blow up of $X$ along a regular center $Y$ contained  in $F_s(X)$. This will be called an \emph{$F_s$-permissible blow up}, or simply a \emph{permissible blow up} when there is no confusion with $s$. In this case we will also say that $Y$ is an \textit{$F_s$-permissible center}.
\item[ii)] An open restriction, i.e., $X_1$ is an open subscheme of $X$. In order to avoid trivial transformations, we will always require $X_1 \cap F_s(X) \neq \emptyset$.
\item[iii)] The multiplication of $X$ by an affine line, say $X_1 = X \times \mathbb{A}^1_k$.
\end{enumerate}
Note that, in either case, $\max\mult(X) \geq \max\mult(X_1)$  (see \cite{D}). We will say that a sequence of transformations, say
\[ \xymatrix {
	X = X_0 &
	X_1 \ar[l]_-{\varphi_1} &
	\dots \ar[l]_-{\varphi_2} &
	X_n \ar[l]_-{\varphi_{n}}
}, \]%
is an \textit{$F_s$-local sequence on $X$} if $\varphi_i$ is an $F_s$-local transformation of $X_{i-1}$ for $i = 1, \dotsc, n$, and
\[
	s = \max\mult(X_0)
	= \dotsc = \max\mult(X_{n-1})
	\geq \max\mult(X_n) .
\]%
In the following lines we will study the behavior of transversality under local sequences. The main result is Theorem \ref{thm:Frs-sequence}.

\begin{remark}[multiplicity and \'etale morphisms]
\label{rmk:multiplicity-etale}
The behavior of the multiplicity is naturally compatible with \'etale topology. Namely, given an equidimensional scheme $X$ of finite type over $k$ and an \'etale morphism $\varepsilon : \widetilde{X} \to X$, one has that $\mult(\tilde{\xi}) = \mult\bigl(\varepsilon(\tilde{\xi})\bigr)$ for every $\tilde{\xi} \in \widetilde{X}$. Moreover, every $F_s$-local sequence over $X$, say
\[ \xymatrix {
 	X &
 		X_1 \ar[l] &
 		\dotsb \ar[l] &
 		X_l, \ar[l]
} \]%
induces by base change an $F_s$-local sequence over $\widetilde{X}$, say
\[ \xymatrix {
 	\widetilde{X} &
 		\widetilde{X}_1 \ar[l] &
 		\dotsb \ar[l] &
 		\widetilde{X}_l, \ar[l]
} \]%
and a commutative diagram,
\[ \xymatrix @R=15pt {
 	\widetilde{X} \ar[d]^{\varepsilon} &
 		\widetilde{X}_1 \ar[l] \ar[d]^{\varepsilon_1} &
 		\dotsb \ar[l] &
 		\widetilde{X}_l \ar[l] \ar[d]^{\varepsilon_l} \\
 	X &
 		X_1 \ar[l] &
 		\dotsb \ar[l] &
 		X_l, \ar[l]
} \]%
where each $\varepsilon_i$ is \'etale.
\end{remark}

Consider a finite morphism of schemes, say $\beta : X' \to X$, a closed center $Y' \subset X'$, and its image in $X$, say $Y = \beta(Y')$. In general, there is no natural map from the blow up of $X'$ at $Y'$ to the blow up of $X$ at $Y$. The next lemma provides a condition under which such map exists, and moreover it is finite.

\begin{lemma}  \label{lm:blow-up-finite}
	Let $\gamma : B \to B'$ be a homomorphism of Noetherian rings. Consider two ideals, say $I \subset B$, and $J \subset B'$, such that $\gamma(I)B'$ is a reduction of $J$. Let $X\to \Spec(B) $ be  the blow at $I$ and let $X' \to \Spec(B')$  be the blow up at $J$. Then there exists a unique morphism $\beta_1 : X'  \to X$ which makes the following diagram commutative:
	\begin{equation} \label{diag:blow-up-finite}
	\xymatrix@R=15pt{
		\Spec(B') \ar[d] &
	X'  \ar[l] \ar[d]^{\beta_1} \\
		\Spec(B) &
		X. \ar[l]
	} \end{equation}%
	In addition, if $B'$ is finite over $B$, then $\beta_1$ is also finite.
\end{lemma}
\begin{proof}
	Since $J$ is integral over $\gamma(I)B'$, we have $\gamma(I) J^n = J^{n+1}$ for some $n \geq 0$ (see \cite[p.~156]{NR}, and \cite[Lemma~1.1, p.~792]{Lipman1}). Therefore  $\gamma(I) J^n \mathcal{O}_{X'} = J^{n+1} \mathcal{O}_{X'}$.  Observe that, by definition, $J \mathcal{O}_{X'}$ is an invertible sheaf of ideals over ${X'}$. Hence $J^n \mathcal{O}_{X'}$ is also invertible, and, in this way, $\gamma(I) J^n \mathcal{O}_{X'} = J^{n+1} \mathcal{O}_{X'}$ implies $\gamma(I) \mathcal{O}_{X'} = J \mathcal{O}_{X'}$. In particular, this means that $\gamma(I) \mathcal{O}_{X'}$ is invertible over ${X'}$.  Thus, by the universal property of the blow up at $I$,  $X\to \Spec(B)$, there exists a unique morphism of schemes, say $\beta_1 : {X'}\to X$, making \eqref{diag:blow-up-finite} commutative.
	
	For the second claim, fix a set of generators of $I$ over $B$, say $I = \langle f_1, \dotsc, f_r\rangle$. Recall that $X = \Proj(B[IW])$, and ${X'}=\Proj(B'[JW])$. From the existence of $\beta_1$, we deduce that ${X'}$ is covered by the affine charts associated to $\gamma(f_1), \dotsc, \gamma(f_r)$, i.e., those given by the rings $\left[ (B'[JW])_{\gamma(f_i) W}\right]_0$ (the homogeneous part of degree $0$ of $(B'[JW])_{\gamma(f_i) W}$). Since $B'$ is finite over $B$, and $\gamma(I)B'$ is a reduction of $J$, the graded algebra $B'[JW]$ is finite over $B[IW]$. Hence $\left[ (B'[JW])_{\gamma(f_i) W}\right]_0$ is finite over $\left[ (B[IW])_{f_i W}\right]_0$ for $i = 1, \dotsc, r$, and therefore $\beta_1$ is finite.
\end{proof}

\begin{remark} \label{rmk:cartas_blowup}
Let $B \subset B'$ be a finite extension of Noetherian domains, and consider two ideals $I \subset B$ and $J \subset B'$ so that $J$ is integral over $IB'$. Let ${X'}\to \Spec(B')$ be the blow up at $J$ and let $X\to \Spec(B)$ be the blow up at $I$.  Fix generators of $I$ and $J$, say $I = \langle x_1, \dotsc, x_r \rangle$, and $J = \langle x_1, \dotsc, x_r, \theta_1, \dotsc, \theta_s \rangle$, where $\theta_1, \dotsc, \theta_s$ are integral over $IB'$. Under these hypotheses, the Lemma says there is a natural finite map ${X'} \to X$. Note that $X$ can be covered by $r$ affine charts of the form $\Spec(B_1), \dotsc, \Spec(B_r)$, with $B_i = B \left [ \frac{x_1}{x_i}, \dotsc, \frac{x_r}{x_i} \right ]$.
Moreover, from the second part of the proof it follows that ${X'}$ can be covered by $r$ affine charts of the form $\Spec(B'_1), \dotsc, \Spec(B'_r)$, where $B'_i = B' \left [ \frac{x_1}{x_i}, \dotsc, \frac{x_r}{x_i}, \frac{\theta_1}{x_i}, \dotsc, \frac{\theta_s}{x_i} \right ]$, and the map ${X'} \to X$ is locally given by the extension $B_i \subset B'_i$, which is finite.
\end{remark}

\begin{theorem}  \label{thm:Frs-sequence}
Let $X$ be a variety with maximum multiplicity $s$ and let $\beta : X' \to X$ be a transversal morphism of generic rank $r$. Then:
\begin{enumerate}[i\textup{)}]

\item An $F_{rs}$-permissible center on $X'$, say $Y' \subset F_{rs}(X')$, induces an $F_s$-permissible center on $X$, say $Y = \beta(Y') \subset F_s(X)$, and a commutative diagram of blow ups   of $X$ at $Y$, say $X\leftarrow X_1$, and of $X'$ at $Y'$, say $X'\leftarrow X_1'$, as follows,
\[ \xymatrix@R=15pt{
	X'  \ar[d]^\beta &
	X'_1   \ar[l] \ar[d]^{\beta_1} \\
	X  &
	X_1, \ar[l]
} \]%
where $\beta_1$  is finite of generic rank $r$. In addition, if $F_{rs}(X'_1) \neq \emptyset$, then $F_s(X_1) \neq \emptyset$, and the morphism $\beta_1$ is transversal.
		
\item Any sequence of $F_{rs}$-permissible blow ups on $X'$, 
		$X' \leftarrow 
		X'_1  
		\dotsb    
		\leftarrow X'_{N-1}  \leftarrow 
		X'_N,$
induces a sequence of $F_s$-permissible blow ups on $X$, and a commutative diagram as follows,
\begin{equation*} %\label{eq:Frs-finite-diagram}
\xymatrix@R=15pt{
	X' \ar[d]^{\beta} &
		X'_1 \ar[l] \ar[d]^{\beta_1} &
		\dotsb \ar[l] &
		X'_{N-1} \ar[l] \ar[d]^{\beta_{N-1}} &
		X'_N \ar[l] \ar[d]^{\beta_N} \\
	X &
		X_1 \ar[l] &
		\dotsb \ar[l] &
		X_{N-1} \ar[l] &
		X_N, \ar[l]
} \end{equation*}%
where each $\beta_i$ is finite of generic rank $r$. Moreover, if $F_{rs}(X'_N) \neq \emptyset$, then $F_s(X_N) \neq \emptyset$, and the morphism $\beta_N$ is transversal.
\end{enumerate}
\end{theorem}

\begin{proof}
Property ii) follows from i) by induction on $N$. Thus we just need to prove i).

Note that the center $Y = \beta(Y')$ is $F_s$-permissible by \ref{resumen}. The varieties $X$ and $X'$ can be locally covered by affine charts of the form $\Spec(B)$ and $\Spec(B')$, with the morphism $\beta$ being given by a finite extension $B \subset B'$. Let $P \subset B$ and $P' \subset B'$ denote the ideals of definition of $Y$ and $Y'$ respectively. By Lemma~\ref{lm:primera_pres}, $P'$ is integral over the extended ideal $PB'$. Then, under these hypotheses, Lemma~\ref{lm:blow-up-finite} says that there is natural finite morphism $\beta_1 : X'_1 \to X_1$ which makes the diagram in i) commutative. Since the blow up of an integral scheme along a proper center is birational, the generic rank of $\beta_1$ coincides with that of $\beta$. That is, $\beta_1$ is a finite morphism of generic rank $r$, which proves the first part of i).

For the second part of the claim, recall first that the multiplicity does not increase when blowing up along permissible centers, and hence $\max\mult(X_1) \leq s$, and $\max\mult(X'_1) \leq rs$. In addition, the map $\beta : X'_1 \to X_1$ has generic rank $r$. Thus $F_{rs}(X'_1) \neq \emptyset$ implies $F_s(X_1) \neq \emptyset$. In particular, if $F_{rs}(X'_1) \neq \emptyset$, then the morphism $\beta : X'_1 \to X_1$ is transversal.
\end{proof}

\begin{remark}  \label{rmk:Frs-local-seqs} Consider a transversal morphism $\beta : X' \to X$ as in the Theorem. As pointed out at the beginning of the section, we are interested in studying the behavior of $F_{rs}(X')$ and $F_s(X)$. Since the set $F_{rs}(X')$ is homeomorphic to its image in $F_s(X)$ via $\beta$, for any open subset $U' \subset X'$, one can find an open subset $U \subset X$ satisfying $U' \cap F_{rs}(X') = \beta^{-1}(U) \cap F_{rs}(X')$. In this setting we will say that $U' \subset X'$ is an $F_{rs}$-permissible restriction if there is an open subset $U \subset X$ so that $U' = \beta^{-1}(U)$. Thus, as a consequence of Theorem~\ref{thm:Frs-sequence}, one readily checks that every $F_{rs}$-local sequence on $X'$, say $X' \leftarrow X'_1 \leftarrow \dotsb \leftarrow X'_N$, induces an $F_s$-local sequence on $X$, and a commutative diagram as follows,
\begin{equation*} %\label{eq:Frs-finite-diagram}
\xymatrix@R=15pt{
 	X' \ar[d]^{\beta} &
 		X'_1 \ar[l] \ar[d]^{\beta_1} &
 		\dotsb \ar[l] &
 		X'_{N-1} \ar[l] \ar[d]^{\beta_{N-1}} &
 		X'_N \ar[l] \ar[d]^{\beta_N} \\
 	X &
 		X_1 \ar[l] &
 		\dotsb \ar[l] &
 		X_{N-1} \ar[l] &
 		X_N, \ar[l]
} \end{equation*}%
where each $\beta_i$ is finite of generic rank $r$. In addition, if $F_{rs}(X'_N) \neq \emptyset$, then $\beta_N$ is transversal.
\end{remark}

\subsection*{Normalization and transversal morphisms}

Let $X$ be a singular variety over a field $k$ with maximum multiplicity $s \geq 2$. The normalization of $X$, say $\overline{X}$, is endowed with a natural finite morphism $\beta : \overline{X} \to X$ which is dominant and birational. Hence, by Zariski's formula (Theorem~\ref{MultForm}), $\max\mult(\overline{X}) \leq \max\mult(X)$.
%\[
%	\max\mult(\overline{X}) \leq \max\mult(X).
%\]%
In addition, if the equality holds, then $\beta : \overline{X} \to X$ is transversal, and $F_s(\overline{X})$ is mapped homeomorphically to $\beta(F_s(\overline{X})) \subset F_s(X)$ (see Corollary~\ref{homeomorphism} and \ref{resumen}).

Assume that the morphism $\beta : \overline{X} \to X$ is transversal, i.e., that $\max\mult(\overline{X}) = \max\mult(X)$. In this case, Theorem~\ref{thm:Frs-sequence} says that any sequence of blow ups along regular equimultiple centers  on $\overline{X}$ induces a sequence of blow ups on $X$. As a consequence of this result, one can also establish a relation between sequences of normalized blow ups on $\overline{X}$ and sequences of blow ups on $X$, as long as transversality is preserved. Recall that the normalized blow up of $\overline{X}$ along a closed center $\overline{Y} \subset \overline{X}$ is the normalization of the blow up of $\overline{X}$ along $\overline{Y}$.

\begin{corollary}
\newcommand{\X}{\overline{X}}
\newcommand{\Y}{\overline{Y}}
Let $X_0$ be a variety over a field $k$, and let $\X_0$ denote the normalization of $X_0$. Assume that $\max\mult(X_0) = \max\mult(\X_0)$. Let $\X_0 \leftarrow \X_1 \leftarrow \dotsb \leftarrow \X_{l-1} \leftarrow \X_l$
be a sequence of normalized blow ups along closed regular centers $\Y_i \subset \Max\mult(\X_{i-1})$, such that
\[
	\max\mult(\X_0)
	= \dotsb
	= \max\mult(\X_{l-1})
	\geq \max\mult(\X_l) .
\]%
Then there is a natural sequence of blow ups on $X_0$, say $X_0 \leftarrow X_1 \leftarrow \dotsb \leftarrow X_{l-1} \leftarrow X_l$,
along closed regular centers $Y_i \subset \Max\mult(X_i)$, so that
\[
	\max\mult(X_0)
	= \dotsb
	= \max\mult(X_{l-1})
	\geq \max\mult(X_l),
\]%
and $\X_i$ is the normalization of $X_i$ for $i = 1, \dotsc, l$.
\end{corollary}

\begin{proof}
	\newcommand{\X}{\overline{X}}
	\newcommand{\Y}{\overline{Y}}
	It suffices to prove the case $l=1$ (the general case follows by induction on $l$). Since $\max\mult(X_0) = \max\mult(\X_0)$, the natural morphism $\beta_0 : \X_0 \to X_0$ is transversal, and hence $Y_0 = \beta_0(\Y_0) \subset \Max\mult(X_0)$ defines a regular center in $X_0$. Let $X_1\to X$ be the blow up at $Y_0$ and let  $\widetilde{{X}_1}\to \overline{X}_0$ be the blow up at $\overline{Y}_0$. Then  by Theorem~\ref{thm:Frs-sequence}, there is a birational map, say $\widetilde{{X}_1} \to X_1$,
which is finite. Let $\X_1$ be the normalization of $\widetilde{{X}_1}$.  Then  it follows that $\X_1$ is the normalization of $X_1$. Moreover, if $\max\mult(\X_1) = \max\mult(\X_0)$, then $\max\mult(X_1) = \max\mult(X_0)$, and the morphism $\beta_1 : \X_1 \to X_1$ is again transversal.
\end{proof}

\subsection*{Strongly transversal morphisms} Let $\beta : X' \to X$ be transversal morphism of generic rank $r$ and suppose that $X$ has maximum multiplicity $s \geq 2$. Under these hypotheses, Theorem~\ref{thm:Frs-sequence} says that any sequence of $F_{rs}$-permissible blow ups on $X'$ induces a sequence of $F_s$-permissible blow ups on $X$. In general, the converse to this Theorem, by blowing up centers over $X$, fails because $\beta$ does not map $F_{rs}(X')$ surjectively to $F_s(X)$. Hence not every center contained in $F_s(X)$ induces a center in $F_{rs}(X')$. This observation motivates the following definition.

\begin{definition} \label{def:strongly-transversal} 
We will say that a transversal morphism of generic rank $r$,  $\beta : X' \to X$,  is \emph{strongly transversal} if $F_{rs}(X')$ is homeomorphic to $F_s(X)$ via $\beta$, and every $F_{rs}$-local sequence over $X'$, say $X' \leftarrow X'_1 \leftarrow \dotsb \leftarrow X'_N$,
induces an $F_s$-local sequence over $X$ in the sense of Remark~\ref{rmk:Frs-local-seqs}, and a commutative diagram as follows,
\begin{equation} \label{diag:def-strong-transversal}
\xymatrix@R=15pt{
 	X' \ar[d]^{\beta} &
 		X'_1 \ar[l] \ar[d]^{\beta_1} &
 		\dotsb \ar[l] &
 		X'_{N-1} \ar[l] \ar[d]^{\beta_{N-1}} &
 		X'_N \ar[l] \ar[d]^{\beta_N} \\
 	X &
 		X_1 \ar[l] &
 		\dotsb \ar[l] &
 		X_{N-1} \ar[l] &
 		X_N, \ar[l]
} \end{equation}%
where each $\beta_i$ is finite of generic rank $r$ (see Remark~\ref{rmk:Frs-local-seqs} above), and induces a homeomorphism between $F_{rs}(X'_i)$ and $F_s(X_i)$. In this case we will also say that $F_{rs}(X')$ is \emph{strongly homeomorphic} to $F_s(X)$. Note in particular that this definition yields $F_{rs}(X'_N) = \emptyset$ if and only if $F_s(X_N) = \emptyset$.
\end{definition}

\begin{remark}
This definition is equivalent to saying that any $F_s$-local sequence on $X$ induces an $F_{rs}$-local sequence on $X'$, and a commutative diagram like \eqref{diag:def-strong-transversal}. In particular, note that the homeomorphism between $F_{rs}(X')$ and $F_s(X)$ is preserved by transformations.
\end{remark}

\begin{remark} \label{rmk:strongly-transversal-equiv-surjective} 
Note that, in virtue of \ref{resumen}, checking that a finite morphism $\beta : X' \to X$ of generic rank $r$ is strongly transversal is equivalent to showing  that $\beta : X' \to X$ maps $F_{rs}(X')$ surjectively onto $F_s(X)$ (where $s$ denotes the maximum multiplicity of $X$), and that this property is preserved by any $F_s$-local sequence.
\end{remark}

From the point of view of resolution of singularities, $\beta : X' \to X$ is strongly transversal if and only if the processes of lowering the maximum multiplicity of $X'$ and $X$ are equivalent. In the case of varieties over a field of characteristic zero, it is possible to give a characterization of strong transversality in terms of Rees algebras (see Theorem \ref{thm:strong-homeo} in Section~\ref{demo_theo}). In the next two sections we introduce the theory of Rees algebras and elimination.

%%%%%%%%%%%%%%%%%%%%%%%%%%%%%%%%%%%%%%%%
%%%%%%%%%%%%%%%%%%%%%%%%%%%%%%%%%%%%%%%%
%%%%%
%%%%% LO BASICO  SOBRE ALGEBRAS DE REES:  
%%%%% CLAUSURA ENTERA
%%%%% CLAUSURA DIFERENCIAL
%%%%% EQUIVALENCIA DEBIL
%%%%%
%%%%%%%%%%%%%%%%%%%%%%%%%%%%%%%%%%%%%%%%
%%%%%%%%%%%%%%%%%%%%%%%%%%%%%%%%%%%%%%%%

\part{Rees algebras and elimination}  \label{parte_2}

\section{Rees algebras and local presentations of the multiplicity} 
\label{ReesAlgebras}

Let $X$ be a $d$-dimensional  algebraic variety defined over a perfect field $k$. When seeking a resolution of singularities of $X$ 
we may start by constructing a sequence of blow ups along closed regular equimultiple centers, say $X \leftarrow X_1 \leftarrow \ldots \leftarrow X_{m-1} \leftarrow X_m$, so that
\[
	\max\mult(X) = \max\mult(X_1)
	= \dotsb = \max\mult(X_{m-1}) > \max\mult(X_m).
\]%
In general, the maximum multiplicity locus of a variety does not define a regular center. Thus the multiplicity function has to be refined in order to obtain regular centers at each stage of the process. It is in this setting that the machinery provided by Rees algebras comes in handy.
For a more detailed introduction to Rees algebras   and their use in resolution of singularities we refer to \cite{notas_Austria} or  \cite{EV} (we will not go into these detalis here). In this paper we will use Rees algebras in the formulation and the proof of Theorem \ref{thm:strong-homeo} and in the proof of  Theorem \ref{rango_r}. Thus,  we will devote  the present  section and the next  to  review the main definitions and results on Rees algebras that we will be needing  in sections \ref{demo_theo} and \ref{construccion}.

\begin{definition} \label{Reesalg} Let $B$ be a Noetherian ring, and let
$\{I_n\}_{n\in {\mathbb N}}$ be a sequence of ideals in $B$ satisfying the
following conditions: $I_0=B$; and $I_k\cdot I_l\subset I_{k+l}$ for all $l,k\in {\mathbb N}$. The graded subring ${\mathcal G}=\bigoplus_{n\geq 0}I_nW^n$ of
the polynomial ring $B[W]$ is said to be a {\em $B$-Rees algebra}, or a  Rees algebra over $B$, if it is a finitely generated $B$-algebra.  A Rees algebra can be described by giving a finite set of generators, say $\{f_{1}W^{n_1},\ldots,f_{s}W^{n_s}\}$, with $f_{i}\in B$ for $i=1\ldots,s$,  in which case we will write ${\mathcal G}=B[f_{1}W^{n_1},\ldots,f_{s}W^{n_s}]\subset B[W]$.

The notion of Rees algebra extends naturally to schemes. Consider a non-necessarily affine scheme $V$. We will say that a quasi-coherent subsheaf $\G \subset \mathcal{O}_V[W]$ is a Rees algebra over $V$, or simply an $\mathcal{O}_V$-Rees algebra, if $V$ can be covered by affine charts of the form $U = \Spec(B)$, where $\Gamma(U,\G)$ is a Rees algebra over $B$. Note that, in this way, there is a natural one-to-one correspondence between the Rees algebras defined over a ring $B$, and those defined over the affine scheme $\Spec(B)$. Thus sometimes we will abuse our notation, and we will make no distinction between them.
\end{definition}

\begin{parrafo} \label{singularlocus}{\bf The singular locus  of a Rees algebra.}  \cite[1.2]{positive} Let $\G = \bigoplus_{n \in \mathbb{N}} \mathcal{I}_n W^n$ be a Rees algebra over a \emph{regular} scheme $V$. The \emph{singular locus} of $\G$ is defined by
\[
	\Sing(\G)
	:= \bigcap_{n \geq 1} \left\{ \xi \in V \mid \nu_\xi(\mathcal{I}_n) \geq n \right \},
\]%
where $\nu_\xi(\mathcal{I}_n)$ denotes the order of the ideal sheaf $\mathcal{I}_n$ at $\xi$. When $V$ is excellent, $\Sing(\G)$ turns out to be a closed subset of $V$. In addition, when $V$ is affine, put $V = \Spec(S)$, and $\G$ is generated by elements $f_1W^{N_1},\ldots,f_sW^{N_s}$, one has that (see \cite[Proposition 1.4]{EV}):
\[
	\Sing(\G)
	= \bigcap_{i = 1}^s \left\{ \xi \in V
		\mid \nu_\xi(f_i) \geq N_i \right \}.
\]%
\end{parrafo}

\begin{parrafo}\label{Ejemplo_Multiplicidad} 
If $X$ is a variety over a perfect field $k$, then, locally in an (\'etale) neighborhood $U$ of a point of maximum multiplicity, there is an embedding in a smooth variety $V$ over $k$ and an ${\mathcal O}_V$-Rees algebra ${\mathcal G}$ such that $\Max \mult(X) \cap U = \Sing {\mathcal G}$, where $\Max \mult(X)$ denotes the (closed) set of points of maximum multiplicity of $X$.  In fact, this equality holds in a very strong sense, to be defined in \ref{d_representable}. This issue will be treated along this section (see Theorem \ref{Equim_Orlando}, and the discussion in \ref{Main_Ideas_Mult} for more precise details). We shall indicate in \ref{411} that a similar statement can be made for the maximum value of the Hilbert-Samuel function on $X$ (see \cite{Hironaka77}).
\end{parrafo}

\begin{parrafo} \label{weaktransforms} {\bf Transforms of Rees algebras by local sequences.} 
{\em Transforms by permissible blow ups.} Let $V$ be a regular scheme, and let ${\mathcal G}$ be a Rees algebra on $V$. 
A regular closed subscheme $Y\subset V$ is said to be {\em permissible} for ${\mathcal G}=\bigoplus_{n}J_nW^n\subset{\mathcal O}_{V}[W]$ if     $Y\subset
\Sing {\mathcal G}$.  A {\em permissible blow up} is the blow up at a permissible center,  $V\leftarrow V_1$. If 
$H_1\subset V_1$ denotes  the exceptional divisor, then for each
$n\in {\mathbb N}$, one has that $J_n{\mathcal O}_{V_1}=I(H_1)^nJ_{n,1}$ 
for some  sheaf of ideals $J_{n,1}\subset {\mathcal
O}_{V_1}$.  The   {\em  transform of ${\mathcal
G}$}  in $V_1$ is then defined as ${\mathcal G}_1:=\bigoplus_{n}J_{n,1}W^n$ (see \cite[Proposition 1.6]{EV}).

{\em  Transforms by smooth morphisms.} 
Let $V\leftarrow V_1$ be either the restriction to some open subset $V_1$ of $V$ or the multiplication by an affine 
space, say $V_1=V\times_k{\mathbb A}_k^n$. Then we define the transform ${\mathcal G}_1$ of ${\mathcal G}$ in 
$V_1$  as the pull-back of ${\mathcal G}$ in $V_1$. 

A {\em ${\mathcal G}$-local sequence over $V$} is a local sequence over $V$, 
\begin{equation}\label{ABCranso1}
\xymatrix @C=30pt{
(V=V_0,\mathcal{G}=\mathcal{G}_0) & \ar[l]_-{\pi_{0}}  (V_1,\mathcal{G}_1) & \ar[l]_-{\pi_{1}}  \cdots & \ar[l]_-{\pi_{m-1}} (V_m,\mathcal{G}_m),
}
\end{equation}
where each  $\pi_i$ is either a permissible monoidal transformation for ${\mathcal G}_i\subset {\mathcal O}_{V_i}[W]$ (and  ${\mathcal G}_{i+1}$ is the transform of ${\mathcal G}_{i}$ in the sense of  \ref{weaktransforms}), or a smooth morphism (and  ${\mathcal G}_{i+1}$ is   the pull-back of ${\mathcal G}_i$ in $V_{i+1}$).
\end{parrafo}

\begin{definition} \label{local_V}
Let $V$ be a regular scheme, and $\G$ an $\mathcal{O}_V$-Rees algebra. A \emph{resolution} of $\G$ consists on a sequence of permissible blow ups,
say
\[ \xymatrix {
	(V=V_0,\mathcal{G}=\mathcal{G}_0) &
	\ar[l]  (V_1,\mathcal{G}_1) &
	\ar[l]  \cdots &
	\ar[l] (V_m,\mathcal{G}_m),
} \]%
so that $\Sing(\G_m) = \emptyset$.
\end{definition}

\noindent{\bf Local presentations of the multiplicity} 

\medskip

In \ref{weaktransforms} and Definition \ref{local_V} we discussed about local sequences of transformations of a Rees algebra $\G$ over a regular scheme $V$. In the next definition we will fix a singular variety $X$ and will consider local transformations on $X$.

\begin{definition} \label{def:local_X}
We will say that a sequence of transformations $X = X_0 \leftarrow X_1 \leftarrow \ldots \leftarrow X_m$ is a {\em $\mult(X)$-local sequence} if for $i=0,\ldots, m-1$,  $X_i\leftarrow X_{i+1}$ is either blow at a smooth center $Y\subset \Max\mult(X_i)$, or else $X_{i+1}=X_i\times {\mathbb A}_k^1$ and $X_i\leftarrow X_{i+1}$ is the smooth morphism defined by the natural projection.
\end{definition}

\begin{definition}\label{d_representable} 
	We will say that {\em $\mult(X)$ is locally representable for a scheme} $X$, if  for each point $\xi\in \Max\mult X$ there is an open (\'etale)  neighborhood (which we denote again by $X$ to ease the notation),   an embedding  in some regular scheme, say $X \subset V$, and a  ${\mathcal O}_{V}$-Rees algebra ${\mathcal G}$ so that the following conditions hold:
	
	\begin{enumerate}[\indent (1)]
		\item There is an equality, $\Max\mult(X)=\Sing(\G)$, of closed sets;
		
		\item Every $\mult(X)$-local sequence as in Definition~\ref{def:local_X}, say
		$X = X_0 \leftarrow X_1 \leftarrow \ldots \leftarrow X_m$ with 
		\[
		\max\mult(X) = \max\mult(X_0) = \dotsb = \max\mult(X_{m-1}) \geq \max\mult(X_m)
		\]%
		induces a ${\mathcal G}$-local sequence 
		\begin{equation}\label{localG}
		\xymatrix @C=30pt{
			(V=V_0,\mathcal{G}=\mathcal{G}_0) &
			\ar[l]  (V_1,\mathcal{G}_1) &
			\ar[l]  \cdots &
			\ar[l] (V_m,\mathcal{G}_m),
		} \end{equation}%
		and: $\Max\mult(X_i) = \Sing(\G_i)$ for $i=1, \dots, m-1$; if $\max\mult({X}_{m-1}) = \max\mult({X}_{m})$, then
		$\Max\mult({X}_m) = \Sing(\G_m)$; if $\max\mult(X_{m-1}) > \max\mult(X_m)$, then
		$\Sing(\G_m) = \emptyset$.
		
		\item Conversely, any  $\G$-local sequence as \eqref{localG} with $\Sing(\G_i) \neq \emptyset$ for $i=0,\ldots,m-1,$ 
		induces an $\mult(X)$-local sequence $X = X_0 \leftarrow X_1 \leftarrow \ldots \leftarrow X_m$ 
		with
		\[
		\max\mult(X) = \max\mult(X_0) = \dots = \max\mult(X_{m-1}) \geq \max\mult(X_m),
		\]%
		and: $\Max\mult(X_i) = \Sing(\G_i)$ for $i=1, \dots, m-1$;  
		$\max\mult(X_{m-1}) = \max\mult(X_{m})$ if $\Max\mult(X_m) = \Sing(\G_m) \neq \emptyset$, 
		$\max\mult(X_{m-1}) > \max\mult(X_{m})$ if $\Sing(\G_m) = \emptyset$. 
	\end{enumerate}
\end{definition}  % Def. Globally representable

\begin{parrafo}\label{locrep}  It can be proved that the previous condition is fulfilled for varieties defined over a prefect field, at least locally in \'etale topology. Namely, given a variety $X$ over a perfect field $k$, and a point $\xi \in X$, there exists an \'etale neighborhood of $X$ at $\xi$, say $U$, a closed immersion of $U$ into a regular variety $V$, and an $\mathcal{O}_V$-Rees algebra $\G$ satisfying conditions (1), (2) and (3) (see \cite[\S7]{multiplicidad}). This local embedding, together with the Rees algebra $\G$, will be called a \emph{local presentation} of $\Max\mult(X)$. 
\end{parrafo}

\begin{remark} \label{rmk:presentasions-etale}
Hereafter, when we refer to a local presentation of $X$, we shall simply write a pair $(V,\G)$, assuming that a closed immersion $X \hookrightarrow V$ has already been fixed, and omitting the fact that $X \hookrightarrow V$ might by defined in an \'etale neighborhood.
\end{remark}

\begin{theorem} \label{Equim_Orlando} \cite[\S5]{multiplicidad} 
Let $X$ be an equidimensional scheme of finite type over a perfect field $k$, and let $s = \max\mult(X)$. Then, $\mult(X)$ is locally representable.
\end{theorem}

\begin{parrafo}\label{411} {\bf Hilbert-Samuel vs. Multiplicity.} In \cite{Hironaka77} Hironaka proves that the Hilbert-Samuel function of an algebraic variety $X$, or more precisely the stratum of points having the same Hilbert-Samuel function, is locally representable.
There, instead of $\mult(X)$-local sequences, one has to consider $\text{HS}_X$-local sequences: set $\Max\text{HS}_X$ as the closed set of points with maximum Hilbert-Samuel function on $X$. We will say that $X = X_0 \leftarrow X_1 \leftarrow \ldots \leftarrow X_m$ is an $\text{HS}_X$-local sequence if for $0 \leq i \leq m-1$, the morphism $X_i \leftarrow X_{i+1}$ is either a blow up at a smooth center $Y \subset \Max\text{HS}_{X_i}$ (this is normal flatness), or else $X_i \leftarrow X_{i+1}$ is a smooth morphism  with $X_{i+1} = X_i \times {\mathbb A}_k^1$. 

By considering the Hilbert-Samuel function, or the multiplicity function, we ultimately seek to produce a resolution of singularities of $X$. In the first case, we seek for a resolution by blowing up at normally flat centers; in the second case, one uses equimultiple centers. Over fields of characteristic zero, in either of these two approaches, Rees algebras can be used in order to produce such centers (as was done in \ref{locrep} for the case of the multiplicity).
\end{parrafo}

\begin{parrafo} {\bf Some ideas behind the proof of Theorem \ref{Equim_Orlando}.} \label{Main_Ideas_Mult} Here we seek for local representations of the multiplicity on $X$ in the sense of Definition~\ref{d_representable}. Suppose that $X$ is affine, say $X = \Spec(B)$, where $B$ is an equidimensional algebra of finite type over a perfect field $k$. Fix a closed point $\xi \in F_s(X)$. After replacing $X$ by a suitable \'etale neighborhood of $\xi$, we may assume that there exists a regular domain $S \subset B$ so that  $B$ is a finite extension of $S$ of generic rank $s$, which necessarily satisfies property (*) from  \ref{condicion_estrella} at $\xi$. Then, choose $\theta_1, \dotsc, \theta_m \in B$ so that $B = S[\theta_1, \dotsc, \theta_m]$. These  elements induce a surjective morphism $S[Z_1, \dotsc, Z_m] \to B = S[\theta_1, \dotsc, \theta_m]$, where $Z_1, \dotsc, Z_m$ represent variables, and we can take $V = \Spec(S[Z_1, \dotsc, Z_m])$. Observe that this morphism induces a natural embedding of $X$ into $V$. Moreover, since $\theta_i$ is integral over $S$, it has a minimal polynomial over $K$, the field of fractions of $S$.  Denote it by $f_i(Z_i) \in K[Z_i]$, and set $d_i = \deg(f_i)$. It can be shown that the coefficients of these polynomials actually belong to $S$, i.e., that $f_i(Z_i) \in S[Z_i]$. Finally, (see \cite[\S 7.1]{multiplicidad} for details), a local presentation of $F_s(X)$ is given by the Rees algebra:
\[
	\G_V
	= \mathcal{O}_V [f_1(Z_1) W^{d_1}, \dotsc, f_m(Z_m) W^{d_m}]
	\subset \mathcal{O}_V[W].
\]%

\end{parrafo}

{\bf\noindent Comparing different presentations: weak equivalence and canonical representatives.}
 Consider an equidimensional scheme $X$ where $\mult(X)$ is representable via local embeddings, and fix a point $\xi \in \Max\mult(X)$. In principle, there may be many presentations of $\Max\mult(X)$ (locally at $\xi$) coming from different immersions and Rees algebras. As we want to use local presentations to find a resolution of singularities of $X$ in the case of characteristic zero, one would like to show that two different presentations lead to the same process of resolution. Assume that a closed immersion $X \hookrightarrow V$ has been fixed, and that we have two presentations of $\Max\mult(X)$ given by two $\mathcal{O}_V$-Rees algebra $\G$ and $\G'$. Then, 
 we would like that $\G$ and $\G'$ lead to the same process of resolution. As both of them represent $\Max\mult(X)$, we have that $\Sing(\G) = \Max\mult(X) = \Sing(\G')$, and this equality is preserved by local sequences. This motivates the following definition.

\begin{definition} \label{def_weak_equiv}
Two Rees algebras $\G$ and $\G'$ over a regular scheme $V$ are \emph{weakly equivalent} if:
%We will say that two Rees algebras over a regular scheme $V$, say $\G$ and $\G'$,  are \emph{weakly equivalent} if:
\begin{enumerate}[\indent(1)]
\item $\Sing(\G) = \Sing(\G')$.
\item Any $\G$-local sequence on $V$, say $V \leftarrow V_1 \leftarrow \dotsb \leftarrow V_m$, is a $\G'$-local sequence, and vice-versa.
\item For any local sequence as that in (2), if $\G_i$ and $\G'_i$ denote the transforms of $\G$ and $\G'$ on $V_i$ respectively, we have that $\Sing(\G_i) = \Sing(\G'_i)$.
\end{enumerate}
\end{definition}

It is easy to see that weak equivalence defines an equivalence relation on the family of $\mathcal{O}_V$-Rees algebras, and thus each $\mathcal{O}_V$-Rees algebra $\G$ has a class of equivalence. By definition, if an $\mathcal{O}_V$-Rees algebra represents $\Max\mult(X)$ (via a fixed immersion $X \hookrightarrow V$), any other element of its class also do so. It can be proved that two weakly equivalent $\mathcal{O}_V$-Rees algebras induce the same invariants, and hence they share the same process of resolution of Rees algebras (over fields of characteristic zero). In this way one can see that the process of simplification of the maximum multiplicity of a singular variety $X$ is independent of the local presentations chosen (see \cite[\S27]{notas_Austria} or \cite{EV}).

\begin{remark}
\label{rmk:tree-of-transformations}
Consider a Rees algebra $\G$ defined over a smooth variety $V$. 
One can consider a tree with $V$ as its root and consisting on all the $\G$-local sequences as in \eqref{ABCranso1}. This tree will be called the \emph{tree of permissible transformations} of $\G$, and we shall denote it by $\mathcal{F}_V(\G)$. Given another Rees algebra $\G'$ over $V$, if every $\G$-local sequence induces a $\G'$-local sequence, then we shall say that $\mathcal{F}_V(\G) \subset \mathcal{F}_V(\G')$. Note that, with this notation, $\G$ is weakly equivalent to $\G'$ if and only if $\mathcal{F}_V(\G) \subset \mathcal{F}_V(\G')$ and $\mathcal{F}_V(\G') \subset \mathcal{F}_V(\G)$ and, in such case, we shall write $\mathcal{F}_V(\G) = \mathcal{F}_V(\G')$.
\end{remark}

\begin{parrafo} \label{integral_closure_Rees}  {\bf Integral closure.}  
Consider a regular domain $S$, and a $S$-Rees algebra $\G \subset S[W]$. We define the integral closure of $\G$, which we shall denote by $\overline{\G}$, as that of $\G$ regarded as a ring, inside its field of fractions. Since $S$ is normal, $S[W]$ is normal, and hence $\overline{\G} \subset S[W]$. In addition, if $S$ excellent, $\overline{\G}$ is finitely generated over $S$ (see \cite[7.8.3 ii) and vi)]{EGAIV}). Thus $\overline{\G}$ is again a Rees algebra over $S$. As integral closure is a concept of local nature, this notion extends naturally to non-affine schemes. In this way, given a regular excellent scheme $V$, and an $\mathcal{O}_V$-Rees algebra $\G$, we may talk unambiguously about the integral closure of $\G$, say $\overline{\G}$, which is again an $\mathcal{O}_V$-Rees algebra.
\end{parrafo}

\begin{remark}  
Observe that the regularity of $S$ (resp. $V$) does not play an essential role on the previous discussion. Thus the concept of integral closure extends to any Rees algebra defined over a normal excellent domain (resp. normal excellent scheme). 
\end{remark}

\begin{lemma}[{cf. \cite[Proposition~5.4]{EV}}]   
Let $V$ be a regular excellent scheme, and $\G$ an $\mathcal{O}_V$-Rees algebra. Then the integral closure of $\G$, say $\overline{\G}$, is weakly equivalent to $\G$.
\end{lemma}

\begin{parrafo} {\bf Differential Rees algebras.} \label{Dif_Algebras} Let $S$ be a smooth algebra over a perfect field $k$. For any non-negative integer $n$, denote by $\Diff^n_{S|k}$ the module of differential operators of order at most $n$ of $S$ over $k$. A Rees algebra over $S$, say $\G = \bigoplus_{i \in \mathbb{N}} I_i W^i$, is said to be \emph{differentially saturated} (with respect to $k$) if the following condition holds: for any homogeneous element $fW^N \in \G$, and any differential operator $\Delta \in \Diff^n_{S|k}$, with $n < N$, we have that $\Delta(f)W^{N-n} \in \G$. In particular, $I_{i+1} \subset I_i$, since $\Diff^0_{S|k} \subset \Diff^1_{S|k}$. Note that this notion extends naturally to Rees algebras defined on a smooth variety over $k$, say $V$: in this case, we denote by $\Diff^n_{V|k}$ the sheaf of differential operators of $V$ over $k$, and the condition of being differentially saturated is defined locally.

Given a smooth variety over a perfect field $k$, and an arbitrary $\mathcal{O}_V$-Rees algebra $\G$, there is a natural way to construct the smallest differentially saturated algebra containing $\G$ (see \cite[Theorem~3.4]{integraldifferential}), which we will denote by $\Diff(\G)$. It can be checked that $\Diff(\G)$ is again finitely generated over $\mathcal{O}_V$ (cf. \cite[Proof of Theorem~3.4]{integraldifferential}), and hence it is a Rees algebra. Moreover $\Diff(\G)$ is weakly equivalent to $\G$ (see Giraud's Lemma \cite[Theorem~4.1]{EV}).
\end{parrafo}

\begin{theorem}[{\cite{Hironaka05}, \cite[Theorem 3.10]{Indiana}}]
\label{thm:canonical-indiana}
Let $V$ be a regular variety over a perfect field $k$. Then two $\mathcal{O}_V$-Rees algebra $\G$ and $\G'$ are weakly equivalent if and only if $\overline{\Diff(\G)} = \overline{\Diff(\G')}$.
\end{theorem}

In particular, this theorem says that $\overline{\Diff(\G)}$ is the canonical representative of the class of $\G$. A generalization of this result to the case of varieties over non-perfect fields can be found in \cite[Theorem~6.6.8]{TesisCarlos}.

\begin{parrafo} {\bf Relative Differential Rees algebras.} \label{Relative_Dif_Algebras}
Let $\beta: V\to V'$ be a smooth morphism of smooth schemes defined over a perfect field $k$ with $\dim V > \dim V'$. Then, for any integer $s$, the sheaf of relative differential operators of order at most $s$,   $\Diff_{V/V'}^s
$, is locally free over $V$.  We will say that a sheaf of ${\mathcal O}_V$-Rees algebras ${\mathcal G}=\oplus_nI_nW^n$ is a {\em $\beta$-relative differential Rees algebra} or simply a {\em $\beta$-differential Rees algebra} if there is an affine covering $\{U_i\}$ of $V$, such that for  every homogeneous element $fW^N\in \G$ and every $\Delta\in \Diff_{V/V'}^n(U_i)$ with $n<N$,  we have that $\Delta(f)W^{N-n}\in {\mathcal G}$ (in particular, $I_{i+1}\subset I_i$  since  $ \Diff_{V/V'}^0\subset \Diff_{V/V'}^1$). Given an arbitrary Rees algebra $\G$ over $V$  there is a natural way to construct a $\beta$-relative differential algebra with the property of being the smallest containing $\G$,  and we will denote it by $\Diff_{V/V'}(\G)$   (see \cite[Theorem 2.7]{hpositive}). Relative differential Rees algebras will play a role in the definition  of the so called {\em elimination algebras} that will be treated in Section \ref{elimination1}. 
	
\end{parrafo}

\color{black}

\medskip

\noindent {\bf The Rees algebra associated to $\Max\mult(X)$.}

\begin{theorem}  \label{thm:rest_definida_b} \cite[Theorem~5.3]{Carlos}
Let $X$ be a singular variety defined over a perfect field $k$. Suppose that there are two immersions $X \subset V_1$ and $X \subset V_2$ together with two Rees algebras $\G_1$ and $\G_2$ over $V_1$ and $V_2$ respectively such that $(V_1,\G_1)$ and $(V_2,\G_2)$ represent the stratum of maximum multiplicity of $X$ \textup{(}see Definition~\ref{d_representable}\textup{)}. Then the $\mathcal{O}_X$-algebras $\Diff(\G_1) \rvert_{X}$ and $\Diff(\G_2) \rvert_{X}$ are equal up to integral closure.
\end{theorem}

\begin{remark}  
Recall that a situation as in the previous theorem can always be achieved locally in \'etale topology (see \ref{locrep}). Moreover, it can be shown that if $X\leftarrow \widetilde{X}_1$ and $X\leftarrow \widetilde{X}_2$ are two \'etale neighborhoods together with embedings  $\widetilde{X}_1\subset V_1$ and  $\widetilde{X}_2\subset V_2$   where two Rees algebras attached to the maximum  multiplicity locus are defined,  then there is an \'etale (common) neighborhood $\widetilde{X}$ of $X, \widetilde{X}_1,$ and  $\widetilde{X}_2$, and two embeddings, $\widetilde{X}\subset \widetilde{V}_1$ and $\widetilde{X}\subset \widetilde{V}_2$ together with two Rees algebras,  say $\G_1$ and $\G_2$ defined on $\widetilde{V}_1$ and  $\widetilde{V}_2$ respectively,   representing the maximum multiplicity of $\widetilde{X}$. This is the sense  in which the previous theorem should be interpreted (see \cite[\S 29]{notas_Austria}). 
\end{remark}

\begin{definition} \label{def:algebra_restringida_b}
Let $X$ be a singular variety over a perfect field $k$ with maximum multiplicity $s > 1$. Given a local presentation of $\Max\mult(X)$ as in Definition~\ref{d_representable}, say $(V,\G)$, we define the \textit{Rees algebra attached to $\Max\mult(X)$} to be $\G_X := \overline{\Diff(\G) \rvert_X}$.
\end{definition}

\begin{remark} \label{rmk:algebra_restringida_b}
Given a singular variety $X$ over a perfect field $k$, local presentations of $\Max\mult(X)$ exist only in \'etale topology (see \ref{locrep} and  \ref{rmk:presentasions-etale}). Thus $\G_X$ is only defined locally in \'etale topology. Theorem~\ref{thm:rest_definida_b} says that $\G_X$ does not depend on the choice of the local presentation $(V,\G)$.
\end{remark}

%%%%%%%%%%%%%%%%%%%%%%%%%%%%%%%%%%%%%%%%
%%%%%%%%%%%%%%%%%%%%%%%%%%%%%%%%%%%%%%%%
%%%%%
%%%%% ALGEBRAS DE ELIMINACION 
%%%%%  
%%%%%%%%%%%%%%%%%%%%%%%%%%%%%%%%%%%%%%%%
%%%%%%%%%%%%%%%%%%%%%%%%%%%%%%%%%%%%%%%%

\section{Elimination algebras} \label{elimination1}

In the previous section we discussed about local presentations of the multiplicity for a given variety $X$ via some pair $(V,\G)$. This has been useful to prove algorithmic resolution of singularities in characteristic zero using the multiplicity as main invariant. 
In this context, one  advantage  of using the multiplicity instead of the Hilbert-Samuel function (the invariant used by Hironaka in \cite{Hironaka64}) is that, at least when the characteristic is zero, it can be shown that the lowering of the maximum multiplicity of $X$ is equivalent to resolving a Rees algebra defined in some $d$-dimensional regular scheme, where $d$ is the dimension of $X$.
This follows from the fact that the local presentation of the multiplicity is found by considering a suitable finite projection from $X$ to some regular scheme $Z$ (see \ref{Main_Ideas_Mult} and Example~\ref{ex:representation-in-dim-d} below).

The previous discussion leads to the following question. Once a pair $(V, \G)$ is given, one may wonder whether there is another $(V', \G')$ with $\dim V'< \dim V$ somehow {\em equivalent} to $(V, \G)$. That is, we would like, (1) that $\Sing (\G)$ be homeomorphic to $\Sing (\G')$ in some sense, and (2), that this homeomorphism be preserved by local sequences. This would mean that finding a resolution of  $\G$ is {\em  equivalent} to finding a resolution of $\G'$, the latter being {\em less complex} since the problem concerns an ambient space of lower dimension. This would allow us to {\em resolve} Rees algebras by using an inductive argument. This is the motivation behind  the theory of {\em elimination algebras}  treated in this section (further details about elimination algebras can be found in \cite{BrV}, \cite{notas_Austria}, \cite{positive}, and \cite{hpositive}).  

 To define an elimination algebra for a given one, the starting point will be a pair $(V, \G)$ as before, and some smooth morphism  to a (smooth) scheme of lower dimension $\beta: V\to V'$.  To be able to talk about elimination we will require some additional conditions on $\beta$ (which will be made precise in  Definition \ref{def_admisible} below).  Now, as we will see, the dificulty here is that if we are given a $\G$-permissible sequence over $V$ we would   like to be able to define (in a natural way) another over $V'$ together with smooth morphisms $\beta_i$  and commutative diagrams: 
 $$\xymatrix@R=15pt{V \ar[d]_{\beta} & V_1\ar[l] \ar[d]_{\beta_1} & \ar[l] \ldots & \ar[l] V_l \ar[d]_{\beta_l} \\
 	V' & V'_1\ar[l]   & \ar[l] \ldots & \ar[l] V_l'.
 }$$
 
However this requires a little bit of care when the permissible transformations involve permissible blow ups. This motivates the following discussion which we hope will help to clarify the meaning of  Definition \ref{def_admisible}.

\begin{parrafo}{\bf On the compatibility of permissible blow ups  with smooth projections.} \label{repetido} 
Let $\beta: V \to V'$ be a smooth morphism of regular varieties over a perfect field $k$ with $\dim V \geq \dim V'$. Let $\G \subset {\mathcal O}_{V}[W]$ and $\G' \subset {\mathcal O}_{V'}[W]$ be Rees algebras. Suppose that $\Sing ({\G})$ is homeomorphic to $\Sing (\G')$ via $\beta$, and let $Y \subset \Sing \G$ be a permissible center.  If $Z'=\beta(Y) \subset \Sing (\G')$ is also a permissible center, then we can consider the blows up of $V$ and $V'$ at $Y$ and $Z=\beta(Y)$ respectively, say 
\[ \xymatrix@R=15pt{
	V \ar[d]_{\beta} & V_1\ar[l] \\
	V' & V'_1.\ar[l]  
}\]%
Let $\G'_1$ be the transform of $\G'$  in $V'_1$ and let $\G_1$ be the transform of $\G$ in $V_1$. In general, $\beta$ cannot be lifted to a morphism $\beta_1: V_1 \to V'_1$  to complete the square. However for our purposes it would be enough if we could guarantee that $\beta$ can be lifted to a smooth morphims from some open subset $U_1\subset V_1$ with $\Sing (\G_1)\subset U_1$,  
\begin{equation} \label{diagrama_abierto}
\xymatrix@R=15pt{V \ar[d]_{\beta} & U_1\ar[l]  \ar[d]_{\beta_1} \\
	V' & V'_1. \ar[l]   
}
\end{equation}%
In such case $\beta_1(\Sing (\G_1))$ can be defined  and  we may wonder  whether it is homeomorphic to $\Sing (\G'_1)$ or not.  The situation is quite similar if instead we consider a permissible center $Z \subset \Sing (\G')$ and consider the blow ups of $V'$ at $Z$ and of $V$ at $\beta^{-1}(Z)\cap \Sing (\G)$. This motivates the formulation of  part (2) of the following definition. As we will see, for  a   smooth morphism $\beta: V\to V'$ to be  {\em $\G$-admissible} (see Definition \ref{def_admisible}),   we will  only ask for  the existence of a commutative diagram like (\ref{diagrama_abierto}) after a $\G$-permissible blow up  (see Remark \ref{demo_abiertos}). 
	
Thus, in forthcoming discussions whenever we define local sequences over $V'$ and $V$ as above we will be assuming that the blow ups at permissible centers  will be restricted to suitable open subsets of the transforms of $V$ that contain the singular locus of the transforms of $\G$ so that commutative diagrams as (\ref{diagrama_abierto}) can be considered. 
\end{parrafo}

\begin{definition} \label{def_admisible}
Let $V^{(n)}$ an $n$-dimensional smooth variety over a perfect field $k$, and let $\G^{(n)}$ be a Rees algebra over $V^{(n)}$. We will say that a smooth morphism from $V^{(n)}$ to a smooth $k$-variety $V^{(n-e)}$ of dimension $(n-e)$, say $\beta : V^{(n)} \to V^{(n-e)}$, is a \emph{$\G^{(n)}$-admissible projection} if the following conditions hold:
\begin{enumerate}
\item $\beta$ maps $\Sing (\G^{(n)})$ homeomorphically to its image in $V^{(n-e)}$. Moreover, it is required that, for any closed subscheme $Y \subset \Sing (\G^{(n)})$, $Y$ is regular if and only if $\beta(Y) \subset V^{(n-e)}$ is so;

\item For any $\G^{(n)}$-permissible sequence of transformations on $V^{(n)}$, say
\begin{equation} \label{eq:def-Gn-projection:seq-n}
\xymatrix @R=-.5ex {
	\G^{(n)} = \G^{(n)}_0 &
		\G^{(n)}_1 &
		&
		\G^{(n)}_l \\
	V^{(n)} = V^{(n)}_0 &
		V^{(n)}_1 \ar[l]_-{\varphi_1} &
		\cdots \ar[l]_-{\varphi_2} &
		V^{(n)}_l , \ar[l]_-{\varphi_l}
} \end{equation}%
%\[ \xymatrix @R=0pt {
%	(V^{(n)}, \G^{(n)})  & (V^{(n)}_1,\G^{(n)}_1)  \ar[l] & \cdots \ar[l] & (V^{(n)}_l, \G^{(n)}_l) , \ar[l]
%} \]%
there exists a collection of open subschemes, say $U^{(n)}_i \subset V^{(n)}_i$ for $i = 1, \dotsc, l$, together with a sequence of transformations over $V^{(n-e)}$, say
\begin{equation}  \label{eq:def-Gn-projection:seq-n-e}
\xymatrix @R=0pt {
	V^{(n-e)} = V^{(n-e)}_0 &
		V^{(n-e)}_1  \ar[l] &
		\cdots \ar[l] &
		V^{(n-e)}_l , \ar[l]
} \end{equation}%
and a commutative diagram, say
\begin{equation} \label{diag:def-Gn-projection}
\xymatrix@R=10pt {
	V^{(n)} \ar[ddd]^-{\beta} &
		V^{(n)}_1  \ar[l]_-{\varphi_1} &
		V^{(n)}_2  \ar[l]_-{\varphi_2} &
		\cdots \ar[l]_-{\varphi_3} &
		V^{(n)}_l \ar[l]_-{\varphi_l} \\
	&
		U^{(n)}_1  \ar[lu] \ar[dd]^-{\beta_1} \ar@{^(->}[u] &
		U^{(n)}_2  \ar[l]  \ar[dd]^-{\beta_2} \ar@{^(->}[u] &
		\cdots \ar[l] &
		U^{(n)}_l \ar[l] \ar[dd]^-{\beta_l} \ar@{^(->}[u] \\
	\\
	V^{(n-e)}  &
		V^{(n-e)}_1  \ar[l] &
		V^{(n-e)}_2  \ar[l] &
		\cdots \ar[l] &
		V^{(n-e)}_l, \ar[l]
} \end{equation}%
with the following properties:
\begin{enumerate}[i\textup{)}]
\item $\varphi_i\bigl(U^{(n)}_i\bigr) \subset U^{(n)}_{i-1}$, in such a way that $U^{(n)}_{i-1} \leftarrow U^{(n)}_i$ is the natural restriction of $\varphi_i$ to $U^{(n)}_i$ for every $i=2, \dotsc, l$;
\item $\Sing\bigl(\G^{(n)}_i\bigr) \subset U^{(n)}_i$ for every $i = 1, \dotsc, l$;
\item Each $\beta_i : U^{(n)}_i \to V^{(n-e)}_i$ is a smooth morphism mapping $\Sing \bigl( \G^{(n)}_i \bigr)$ homeomorphically to its image in $V^{(n-e)}_i$. Moreover, it is required that, for any closed subscheme $Y_i \subset \Sing \G^{(n)}_i$, $Y_i$ is regular if and only if $\beta_i(Y_i) \subset V^{(n-e)}_i$ is so;
\item The sequence in \eqref{eq:def-Gn-projection:seq-n-e} is the sequence naturally induced by \eqref{eq:def-Gn-projection:seq-n} in the following sense: for each $i=1, \dotsc, l$, if $V^{(n)}_i$ is the blow up of $V^{(n)}_{i-1}$ at a closed center $Y_{i-1} \subset \Sing \bigl( \G^{(n)}_{i-1 }\bigr)$, then $V^{(n-e)}_i$ is the blow up of $V^{(n-e)}_{i-1}$ at $\beta_{i-1}(Y_{i-1})$; if $V^{(n)}_i$ is an open restriction or the multiplication by an affine line of  $V^{(n)}_{i-1}$, then so is $V^{(n-e)}_i$ with respect to $V^{(n-e)}_{i-1}$.
\end{enumerate}%
\end{enumerate}%
\end{definition}%

Given $V^{(n)}$ and $\G^{(n)}$, the question of whether there exists a $\G^{(n)}$-admissible projection to some smooth space of dimension $(n-e)$ arises. The existence of such admissible projections can be characterized in terms of the $\tau$-invariant, an invariant attached to each  closed  point $\xi \in \Sing \G^{(n)}$.

\begin{definition}
Let $V^{(n)}$ be a regular variety over a perfect field $k$ and let $\G^{(n)}$ be a Rees algebra over $V^{(n)}$. Fix a closed point $\xi \in \Sing (\G^{(n)})$, and let $\text{Gr}_{m_{\xi}}({\mathcal O}_{V^{(n)},\xi})$ denote the graded ring of $\mathcal{O}_{V^{(n)},\xi}$. Consider the tangent ideal of $\G^{(n)}$ at $\xi$, say $\In_{\xi} \G^{(n)} \subset \text{Gr}_{m_{\xi}}({\mathcal O}_{V^{(n)},\xi})$, defined as the homogeneous ideal generated by $\In_{\xi}(I_n) := \bigl(I_n+m^{n+1}_{\xi}\bigr) / m^{n+1}_{\xi}$
%\[
%	\In_{\xi}(I_n) := \frac{ I_n+m^{n+1}_{\xi}}{m^{n+1}_{\xi}}
%\]%
for all $n \geq 1$. Note that $\text{Gr}_{m_{\xi}}({\mathcal O}_{V^{(n)},\xi}) \simeq k'[Z_1,\ldots,Z_n]$, where $Z_1, \dotsc, Z_n$ is a basis of the subspace of linear forms of $\text{Gr}_{m_{\xi}}({\mathcal O}_{V^{(n)},\xi})$, say $\text{Gr}^1_{m_{\xi}}({\mathcal O}_{V^{(n)},\xi})$,   and $k'$ is the residue field at $\xi$. The \emph{$\tau$-invariant} is the minimum integer so that there exists a basis of ${Gr}^1_{m_{\xi}}({\mathcal O}_{V^{(n)},\xi})$, say $Y_1, \dotsc, Y_n$, satisfying that $\In_{\xi} \G^{(n)}$ can be generated by polynomials in $k'[Y_1, \dotsc, Y_\tau]$ (see \cite{B}). That is, so that
\[
	\In_{\xi} \G^{(n)}
	= \left ( \In_{\xi} \G^{(n)} \cap k'[Y_1, \dotsc, Y_\tau] \right ) \cdot \text{Gr}_{m_{\xi}}({\mathcal O}_{V^{(n)},\xi}).
\]%
\end{definition}

The $\tau$-invariant appears in \cite{Hironaka64}, and has been thoroughly  studied in \cite{Oda1973}, \cite{Oda1983} and \cite{Oda1987}. See also \cite{Hironaka70},  \cite{kaw} and \cite{B}.

\begin{lemma}[{cf. \cite[\S 4]{hpositive},     \cite[\S8]{BrV}}]\label{existencia}
Let $V^{(n)}$ be a smooth variety of dimension $n$ over a perfect field $k$, and let $\G^{(n)}$ be a Rees algebra over $V^{(n)}$. Let $\xi \in \Sing (\G^{(n)})$   be a closed point. Then, for $e \leq \tau_{\G^{(n)}, \xi}$, one can  construct a $\G^{(n)}$-admissible projection (locally in \'etale topology) to some smooth $k$-variety of dimension $(n-e)$,  $\beta : V^{(n)} \to V^{(n-e)}$.
\end{lemma}

\begin{remark} \label{demo_abiertos} If $\beta: V^{(n)}\to V^{(n-e)}$ is a $\G^{(n)}$-admissible morphism (e.g., as the ones constructed in \cite{hpositive} or \cite{BrV}) and if $Y\subset \Sing (\G^{(n)})$ is a permissible center, then it can be shown that there is a well defined open subset $U_1^{(n)}\subset V_1^{(n)} $ of the blow up of $V^{(n)}$ at $Y$ so that the following diagram of smooth projections and blow ups at $Y$ and at  $Z:=\beta(Y)$ conmmutes: 
\begin{equation} \label{diagrama_abierto_bis}
\xymatrix@R=15pt{
	V^{(n)} \ar[d]_{\beta} & U_1^{(n)}\ar[l]  \ar[d]_{\beta_1} \\
	V^{(n-e)} & V^{(n-e)}_1. \ar[l]   
} \end{equation} 
In fact, the open $U_1^{(n)}$ can be defined as follows. Set $\widetilde{Z}=\beta^{-1}(Z)$ and let $\widetilde{Z}_1$ be the strict transform of $\widetilde{Z}$ in $V_1$. Define $U_1^{(n)}:=V_1^{(n)}\setminus \widetilde{Z}_1$. Then it can be checked that $\Sing (\G_1^{(n)}) \subset U_1^{(n)}$ and that $\beta$ can be lifted to a smooth morphism from $U_1^{(n)}$ (see \cite[\S 6]{hpositive}, \cite[Theorem 9.1]{BrV}).   Just to ease the notation, in what follows,  whenever we consider a $\G^{(n)}$-admissible morphism $\beta: V^{(n)}\to V^{(n-e)}$  and the blow up at a permissible center as above, instead of writing a commutative diagram like (\ref{diagrama_abierto_bis}), we will write  
\begin{equation} \label{diagrama_simple}
\xymatrix@R=15pt{
	V^{(n)} \ar[d]_{\beta} & V_1^{(n)}\ar[l]  \ar[d]_{\beta_1} \\
	V^{(n-e)} & V^{(n-e)}_1, \ar[l]   
} \end{equation} 
understanding that $\beta_1$ may only be defined in an open subset of $V_1^{(n)}$. 
\end{remark}

Roughly speaking, Lemma~\ref{existencia} says that, under some assumptions,   the tree of closed sets defined by $\G^{(n)}$ can be  interpreted as a tree of closed subsets  in some smooth scheme of  lower dimension.  The next question is whether one can find a Rees algebra over $V^{(n-e)}$ which represents this tree of closed sets.

\begin{definition}  \label{def:strongly-linked}
Let $\G^{(n)}$ be a Rees algebra over a smooth variety $V^{(n)}$, and let $\beta : V^{(n)} \to V^{(n-e)}$ be a $\G^{(n)}$-admissible projection. We will say that a Rees algebra $\G^{(n-e)}$ over $V^{(n-e)}$ is \emph{strongly linked} to $\G^{(n)}$ via $\beta$ if the following conditions hold:
\begin{enumerate}
\item $\Sing (\G^{(n-e)}) = \beta \bigl( \Sing (\G^{(n)}) \bigr)$ (note that, since $\beta$ is required to be $\G^{(n)}$-admissible, this implies that a closed center $Y \subset \Sing (\G^{(n)})$ is regular if and only if the corresponding closed subscheme, say $\beta(Y) \subset \Sing (\G^{(n-e)})$, is so);
\item Any $\G^{(n)}$-permissible sequence on $V^{(n)}$, say
\begin{equation} \label{eq:def-strongly-linked:seq-n}
\xymatrix @R=-.5ex {
	\G^{(n)} = \G^{(n)}_0 &
		\G^{(n)}_1 &
		&
		\G^{(n)}_l \\
	V^{(n)} = V^{(n)}_0 &
		V^{(n)}_1 \ar[l] &
		\cdots \ar[l] &
		V^{(n)}_l , \ar[l]
} \end{equation}%
induces a $\G^{(n-e)}$-permissible sequence on $V^{(n-e)}$, say
\begin{equation} \label{eq:def-strongly-linked:seq-n-e}
\xymatrix @R=-.5ex {
	\G^{(n-e)} = \G^{(n-e)}_0 &
		\G^{(n-e)}_1 &
		&
		\G^{(n-e)}_l \\
	V^{(n-e)} = V^{(n-e)}_0 &
		V^{(n-e)}_1 \ar[l] &
		\cdots \ar[l] &
		V^{(n-e)}_l , \ar[l]
} \end{equation}%
and a commutative diagram as \eqref{diag:def-Gn-projection} in Definition~\ref{def_admisible} in such a way that $\beta_i$ is a $\G^{(n)}_i$-admissible projection and
\[
	\Sing  (\G^{(n-e)}_i)
	= \beta_i \left( \Sing (\G^{(n)}_i) \right)
\]%
for every $i = 1, \dotsc, l$ (recall that, by condition (2)~ii) of Definition~\ref{def_admisible}, one has that $\Sing (\G^{(n)}_i) \subset U^{(n)}_i$ for all $i$).
\end{enumerate}%
\end{definition}

\begin{remark} \label{res_equiv}
Let $\G^{(n-e)}$ be a Rees algebra which is strongly linked to $\G^{(n)}$ as in the previous definition via the morphism $\beta : V^{(n)} \to V^{(n-e)}$. Then any closed regular center $Z \subset \Sing (\G^{(n-e)})$ induces a closed regular center $Y \subset \Sing (\G^{(n)})$. Thus, any $\G^{(n-e)}$-permissible sequence of transformations like \eqref{eq:def-strongly-linked:seq-n-e} induces a $\G^{(n)}$-permissible sequence like \eqref{eq:def-strongly-linked:seq-n} and a commutative diagram like \eqref{diag:def-Gn-projection}. In particular  $\Sing (\G^{(n-e)}_l) = \emptyset$ if and only if $\Sing (\G^{(n)}_l) = \emptyset$. Thus it turns out that, if $\G^{(n-e)}$ is strongly linked to $\G^{(n)}$, then a resolution of $\G^{(n)}$ induces a resolution $\G^{(n-e)}$, and vice versa.
\end{remark}

\begingroup

In general, given a Rees algebra $\G^{(n)}$ over a smooth variety $V^{(n)}$ and a $\G^{(n)}$-admissible projection $\beta : V^{(n)} \to V^{(n-e)}$, it may occur that one cannot find algebras over $V^{(n-e)}$ which are strongly linked to $\G^{(n)}$. This leads us to introduce a weaker notion:

 \begin{definition}  \cite[1.25, Definitions 1.42, 4.10, Theorem 4.11]{hpositive} \label{def:an-elimination-algebra}
	Let $V^{(n)}$ be an $n$-dimensional smooth variety over a perfect field $k$, and let $\G^{(n)}$ be a Rees algebra over $V^{(n)}$. Consider a $\G^{(n)}$-admissible projection to a smooth variety $V^{(n-e)}$ of dimension $(n-e)$, say $\beta : V^{(n)} \to V^{(n-e)}$ and assume that    $\G^{(n)}$ is a $\beta$-differential Rees algebra.  Then the ${\mathcal O}_{V^{(n-e)}}$-Rees algebra algebra
	\begin{equation}  \label{eq:def-elimination-algebra}
	\G^{(n-e)}
	:= \G^{(n)}
	\cap \mathcal{O}_{V^{(n-e)}} [W] ,
	\end{equation}%
	or any other Rees algebra over $V^{(n-e)}$ with the same integral closure   is called \emph{an elimination algebra} of $\G^{(n)}$ via $\beta$.
\end{definition}

\begin{remark}\label{elimination_abs}
Note that, if an algebra $\G^{(n)}$ is (absolute) differential, then it is relative differential for any $\G^{(n)}$-admissible projection $\beta : V^{(n)} \to V^{(n-e)}$. Hence, in this case,
\begin{equation}
\label{elimination_diff}
\G^{(n-e)}
= \G^{(n)}
\cap \mathcal{O}_{V^{(n-e)}} [W]
\end{equation}
is also  an elimination algebra of $\G^{(n)}$. However, it may occur that the elimination algebra defined in  (\ref{elimination_diff}) is not integral over the one defined in \ref{def:an-elimination-algebra}. In other words,  given a Rees algebra $\G^{(n)}$ in $V^{(n)}$ and a $\G^{(n)}$-admissible projection, $\beta : V^{(n)} \to V^{(n-e)}$, in general the ${\mathcal O}_{V^{(n-e)}}$-Rees algebras; 
$$\Diff(\G^{(n)})\cap \mathcal{O}_{V^{(n-e)}} [W] \ \ \text{ and } \Diff_{V/V'}(\G^{(n)})\cap \mathcal{O}_{V^{(n-e)}} [W]$$
do not share the same integral closure in $ \mathcal{O}_{V^{(n-e)}} [W]$, and, in general, they  may  not  be weakly equivalent either. 
\end{remark}

\begin{remark}
Given a Rees algebra $\G^{(n)}$, and a $\G^{(n)}$-admissible projection $\beta : V^{(n)} \to V^{(n-e)}$, for geometrical reasons we will be interested in the case in which $\beta(\Sing(\G^{(n)}))\subsetneq  V^{(n-e)}$. In fact, if  $\beta(\Sing(\G^{(n)}))=V^{(n-e)}$  then it can be checked that a resolution of $\G^{(n)}$ can be achieved in a trivial way (see \cite[Lemma 13.2]{BrV}). However, in this paper,  because of the hypotheses under we will be considering the admissible projections and the elimination algebras,  we will always  be in the case in which  $\beta(\Sing(\G^{(n)}))\subsetneq  V^{(n-e)}$.   
\end{remark}

\begin{parrafo}\label{imagen_sing}{\bf Elimination algebras and singular loci of Rees algebras.} \cite[Theorem 2.9]{VillaCordoba}, \cite[Corollary 4.12]{hpositive}
Let $V^{(n)}$ be an $n$-dimensional smooth variety over a perfect field $k$, and let $\G^{(n)}$ be a Rees algebra over $V^{(n)}$. Consider a $\G^{(n)}$-admissible projection to a smooth variety $V^{(n-e)}$ of dimension $(n-e)$, say $\beta : V^{(n)} \to V^{(n-e)}$, and suppose that $\G^{(n)}$ is a $\beta$-differential Rees algebra. Define an elimination algebra  $\G^{(n-e)}$ as in (\ref{eq:def-elimination-algebra}). Then: 
\begin{itemize}
	\item[(i)] If the characteristic of the base field $k$ is zero,   $\beta(\Sing(\G^{(n)}))=\Sing(\G^{(n-e)})$; 
	\item[(ii)] If the characteristic of the base field $k$ is positive,  there is a containment $\beta(\Sing(\G^{(n)}))\subseteq\Sing(\G^{(n-e)})$ which may be strict; on the other hand,  if   $\G^{(n)}$ is a differential Rees algebra,  then it can be shown that $\beta(\Sing(\G^{(n)}))=\Sing(\G^{(n-e)})$. However, in general, this equality is not stable after considering local sequences; this issue will be discussed in the following paragraphs. 
	\end{itemize} 
\end{parrafo}

\begin{parrafo}
	\label{eliminacion_blow_ups} {\bf Elimination algebras and permissible transformations.} \cite[Lemma 1.7]{VillaCordoba}, \cite[\S 6]{hpositive}, \cite[\S8, Theorem 9.1]{BrV} Let $V^{(n)}$ be an $n$-dimensional smooth variety over a perfect field $k$, and let $\G^{(n)}$ be a Rees algebra over $V^{(n)}$. Consider a $\G^{(n)}$-admissible projection to a smooth variety $V^{(n-e)}$ of dimension $(n-e)$, say $\beta : V^{(n)} \to V^{(n-e)}$, and suppose that $\G^{(n)}$ is a $\beta$-differential Rees algebra. Define an elimination algebra  $\G^{(n-e)}$ as in (\ref{eq:def-elimination-algebra}). Let $Y\subset \Sing(\G^{(n)})$ be a permissible center. Then it can be shown that $\beta(Y)\subset \Sing(\G^{(n-e)})$ is a $\G^{(n-e)}$-permissible center, and  there is a commutative diagram of blow ups, smooth morphisms and transforms of Rees algebras   
\begin{equation} \label{diagrama_simple_2}
\xymatrix@R=15pt{
\G^{(n)}  & 	V^{(n)} \ar[d]_{\beta} & V_1^{(n)}\ar[l]  \ar[d]_{\beta_1} & \G^{(n)}_1 \\
\G^{(n-e)}  & 	V^{(n-e)} & V^{(n-e)}_1 \ar[l] &  \G^{(n-e)}_1 
} \end{equation} 
as in (\ref{diagrama_simple})  (see Remark \ref{demo_abiertos}),  and:  
\begin{itemize}
	\item[(i)] The morphism $\beta_1:  V_1^{(n)}\to  V_1^{(n-e)}$ is $\G^{(n)}_1$-admissible; 
	\item[(ii)] The Rees algebra $\G^{(n)}_1$ is  a $\beta_1$-differential Rees algebra, therefore an elimination algebra can be defined as in   (\ref{eq:def-elimination-algebra}); 
	\item[(iii)] The transform of $\G^{(n-e)}$ in $V^{(n-e)}_1$,  $\G^{(n-e)}_1$,  is an elimination algebra of  $\G^{(n)}_1$, i.e., up to integral closure, 
	$$ \G^{(n-e)}_1=\G^{(n)}_1\cap {\mathcal O}_{V^{(n-e)}_1}[W].$$
\end{itemize}  
	
\end{parrafo}

\begin{remark}
	With the same setting and notation as in \ref{eliminacion_blow_ups}  above, we would like to point out that, even if we assume that the  Rees algebra $\G^{(n)}$ is a differential Rees algebra, after a permissible blow up,  its transform $\G_1^{(n)}$, in general,  is not a differential Rees algebra; however it is always a $\beta_1$-relative differential Rees algebra.  
\end{remark}

\begin{parrafo} \label{rmk:elim-alg-lower-dim} {\bf Elimination algebras and local sequences.} \cite[\S 16]{notas_Austria}
	Let $V^{(n)}$ be an $n$-dimensional smooth variety over a perfect field $k$, and let $\G^{(n)}$ be a Rees algebra over $V^{(n)}$. Consider a $\G^{(n)}$-admissible projection to a smooth variety $V^{(n-e)}$ of dimension $(n-e)$, say $\beta : V^{(n)} \to V^{(n-e)}$ and suppose that $\G^{(n)}$-is a $\beta$-differential Rees algebra. Define an elimination algebra  $\G^{(n-e)}$ as in (\ref{eq:def-elimination-algebra}). Then it can be checked that any $\G^{(n)}$-local sequence induces a $\G^{(n-e)}$-local sequence and a commutative diagram of smooth  morphims and transforms: 
\begin{equation} \label{diag:elim-alg-lower-dim}
	%\xymatrix @R=0ex {
	%	\G^{(n)} & \G^{(n)}_{1} & & \G^{(n)}_{l} \\
	%	V^{(n)} \ar[dd]^-{\beta}  & V^{(n)}_1   \ar[l] \ar[dd]^-{\beta_1}   &  \cdots \ar[l] &  V^{(n)}_l \ar[l]  \ar[dd]^-{\beta_l} \\
	%	\big. \\
	%	V^{(n-e)}  & V^{(n-e)}_1  \ar[l] & \cdots \ar[l] &  V^{(n-e)}_l, \ar[l] \\
	%	\G^{(n-e)} & \G^{(n-e)}_{1} & & \G^{(n-e)}_{l}
	%}
	\xymatrix @R=0ex {
		(V^{(n)}, \G^{(n)}) \ar[dd]^-{\beta}  & (V^{(n)}_1, \G^{(n)}_{1} )   \ar[l] \ar[dd]^-{\beta_1}   &  \cdots \ar[l] &  (V^{(n)}_l, \G^{(n)}_{l})  \ar[l]  \ar[dd]^-{\beta_l} \\
		\big. \\
		(V^{(n-e)}, \G^{(n-e)}) & (V^{(n-e)}_1, \G^{(n-e)}_{1}) \ar[l] & \cdots \ar[l] &  (V^{(n-e)}_l,  \G^{(n-e)}_{l}) \ar[l] \\
	}
	\end{equation}%
	where each $\beta_i$ is $\G^{(n)}_i$-admissible,  $\G^{(n-e)}_i$  is an elimination algebra of $\G^{(n)}$, and $\beta_i \bigl( \Sing \G^{(n)}_i \bigr) \subseteq \Sing \G^{(n-e)}_i$.
	
	When the characteristic of the base field is zero, then the equality   $\beta_i \bigl( \Sing \G^{(n)}_i \bigr)=\Sing \G^{(n-e)}_i$ holds for $i=1,\ldots, n$, and therefore, in this case the elimination algebra $\G^{(n-e)}$ is stronly linked to $\G^{(n)}$. As a consequence, a resolution of $\G^{(n)}$ induces a resolution of $\G^{(n-e)}$ and vice versa (see Remark \ref{res_equiv}). 
	
	If the charateristic of the base field is positive, the  elimination algebra $\G^{(n-e)}$ might not be strongly linked to $\G^{(n)}$. Moreover, given $\G^{(n)}$ and $\beta : V^{(n)} \to V^{(n-e)}$ as above, it may occur that there is no Rees algebra over $V^{(n-e)}$ which is strongly linked to $\G^{(n)}$ (see Example~\ref{ex:elim-not-strongly} below).
	Still, it can be shown that $\G^{(n-e)}$ is the largest algebra over $V^{(n-e)}$ satisfying the containment 
	\begin{equation}  \label{eq:def-elimination-trees}
	\mathcal{F}_{V^{(n)}} \bigl( \G^{(n)} \bigr)
	\subset \mathcal{F}_{V^{(n)}} \bigl(\beta^* (\G^{(n-e)}) \bigr).
	\end{equation}
	 (see Theorem~\ref{thm:canonical-indiana}). Thus, an elimination algebra of $\G^{(n)}$ can also be regarded as a maximal Rees algebra over $V^{(n-e)}$ with the previous property.
\end{parrafo}

\begin{remark} \label{rmk:elim-alg-characteristic}
Let $\G^{(n)}$ be a Rees algebra over a regular variety $V^{(n)}$ defined over a perfect field $k$, and consider a $\G^{(n)}$-admissible projection, say $\beta : V^{(n)} \to V^{(n-e)}$. Assume in addition that  $\G^{(n)}$ is a differential Rees algebra, and consider the corresponding  elimination algebra   on $V^{(n-e)}$. Now 
suppose that $\mathcal{K}^{(n-e)}$ is a Rees algebra over $V^{(n-e)}$ which is strongly linked to $\G^{(n)}$. In such case, $\mathcal{K}^{(n-e)}$ is, necessarily, an elimination algebra of $\G^{(n)}$ (this follows from the definition of elimination algebra, and Theorem~\ref{thm:canonical-indiana}). 
\end{remark}

Now  suppose that  $\G^{(n)}$  and $\H^{(n)}$ are two weakly equivalent Rees algebras over $V^{(n)}$, and that  $\beta:  V^{(n)} \to V^{(n-e)}$ is $\G^{(n)}$-admissible.  Then it can be shown that      $\beta$ is also  $\H^{(n)}$-admissible (see \cite[\S 5, 6]{B}), and it is natural to ask about the relation between the  elimination algebras.  The next results address these questions:

\begin{theorem}[{cf. \cite[Theorem~4.11]{hpositive}}]
Let $V^{(n)}$ be a regular variety over a perfect field $k$, and let $\G^{(n)}$ be a Rees algebra over $V^{(n)}$. Consider a $\G^{(n)}$-admissible projection $\beta : V^{(n)} \to V^{(n-e)}$, and assume that  $\G^{(n)}$ is $\beta$-differential. Let $\mathcal{K}^{(n)}$ be another Rees algebra over $V^{(n)}$, so that  $\G^{(n)}\subset \mathcal{K}^{(n)}$  and so that the inclusion is   finite. Then $\mathcal{K}^{(n)} \cap \mathcal{O}_{V^{(n-e)}} [W]$ is finite over $\G^{(n)} \cap \mathcal{O}_{V^{(n-e)}} [W]$.
\end{theorem}

\begin{theorem}  [{cf. \cite[Corollary~4.14]{hpositive}}] \label{crl:elim-alg-without-closure} 
Let $V^{(n)}$ be a regular variety over a perfect field $k$, and let $\G^{(n)}$ be a Rees algebra over $V^{(n)}$. Consider a $\G^{(n)}$-admissible projection $\beta : V^{(n)} \to V^{(n-e)}$ and suppose that  $\G^{(n)}$ is a differential Rees algebra. Then the elimination algebra 
\[
	\G^{(n-e)} = \G^{(n)} \cap \mathcal{O}_{V^{(n-e)}} [W] .
\]%
is  also  a differential Rees algebra.
\end{theorem}

The previous theorems together with the discussion in \ref{rmk:elim-alg-lower-dim} lead us to the following conclusion: when the characteristic of the base field is zero, if two Rees algebras   are weakly equivalent, after considering an  admissible projection their respective  elimination algebras  are weakly equivalent; in positive characteristic this fails to hold in general. 

\endgroup

\

The following example illustrates that, in positive characteristic, an elimination algebra  might not be strongly linked to the original Rees algebra.

\begin{example} \label{ex:elim-not-strongly}
Suppose that $k$ is a perfect field of characteristic $2$. Consider the curve $X = \Spec(k[x,y] / \langle y^2-x^3\rangle)$, endowed with its natural immersion in $V^{(2)} = \Spec(k[x,y])$. One can check that
\[
	\G^{(2)}
	= \mathcal{O}_{V^{(2)}} \left [ x^2 W, (y^2 - x^3) W^ 2 \right ]
\]%
is a differential algebra over $V^{(2)}$ which represents $F_2(X)$. Moreover, the inclusion $k[x] \subset k[x,y]$ induces a $\G^{(2)}$-admissible projection of $V^{(2)}$ to $V^{(1)} = \Spec(k[x])$, say $\beta : V^{(2)} \to V^{(1)}$. Thus, by Corollary~\ref{crl:elim-alg-without-closure}, we see that $\G^{(1)} = \mathcal{O}_{V^{(1)}} \left [ x^2 W \right ]$ is an elimination algebra of $\G^{(2)}$.   However, by blowing up $V^{(2)}$ and $V^{(1)}$ at the origin, one can check that $\G^{(1)}$ is not strongly linked to $\G^{(2)}$. To see this recall how Rees algebras transform after a permissible blow up (see \ref{weaktransforms}). Observe that after the blow up, $V^{(2)}\leftarrow V^{(2)}_1$,  at the $k\left[x, \frac{y}{x}\right]$-affine chart, 
$$\G^{(2)}_1=k\left[x, \frac{y}{x}\right] \left[xW, \left(\left(\frac{y}{x}\right)^2-x\right)W^2\right],$$ and therefore, here $\Sing \G^{(2)}_1=\emptyset$. On the other hand, after the blow up $V^{(1)}\leftarrow V^{(1)}_1$ (which is an isomorphism on $V^{(1)}$) $\G^{(1)}_1=k[x][xW]$, and hence $\Sing \G^{(1)}_1\neq \emptyset$. \color{black}
\end{example}

\begin{example}[Representation and elimination in characteristic zero]
\label{ex:representation-in-dim-d} 
Let $X=\text{Spec}(B)$ and $V$ be as in \ref{Main_Ideas_Mult}. Set $V^{(d+m)}:=V$.  With  the same notation  and setting as in \ref{Main_Ideas_Mult},    the Rees algebra  \[
	\G^{(d+m)}
	= \mathcal{O}_{V^{(d+m)}} [f_1(Z_1) W^{d_1}, \dotsc, f_m(Z_m) W^{d_m}]
	\subset \mathcal{O}_{V^{(d+m)}} [W]
\]%
represents $F_s(X)$. Observe that this means that $F_s(X)$ is {\em represented} by $\bigcap_{i=1}^m F_{d_i} (f_i(Z_i))$. (i.e.,  $F_s(X)$ is the intersection of the maximum multiplicity loci of the hypersurfaces defined by the$f_i(Z_i)$ -at least in an \'etale neighborhood of $\xi$). In this setting, it can be proved that the natural projection $\beta : V^{(d+m)} \to V^{(d)} = \Spec(S)$ is $\G^{(d+m)}$-admissible. Moreover, an elimination algebra of $\G^{(d+m)}$ can be obtained by considering suitable functions on the coefficients of the polynomials $f_1(Z_1), \dotsc, f_m(Z_m)$. Namely, when $\characteristic(k) = 0$, one can find elements $a_1, \dotsc, a_m \in S$ so that, after taking the change of variables $Z_i' = Z_i + a_i$,
\[
	f_i(Z_i) = (Z'_i)^{d_i} + b_{i,2} (Z'_i)^{d_i-2} + \dotsb + b_{i,d_i} ,
\]%
with $b_{i,j} \in S$ for all $i,j$. In this case, the $S$-algebra generated by the elements $b_{i,j} W^{j}$, with $1 \leq i \leq m$ and $2 \leq j \leq d_i$, say $\G^{(d)}
	= S[b_{i,j} W^{j}]
	\subset S[W]
$ 
is an elimination algebra of $\G^{(d+m)}$ (see \cite{hpositive} and    \cite[Remark 16.10]{notas_Austria} for further details).
\end{example}

%%%%%%%%%%%%%%%%%%%%%%%%%%%%%%%%%%%%%%%%%%%%%%%%%%%%%%%%%%%%
%%%%%%%%%%%%%%%%%%%%%%%%%%%%%%%%%%%%%%%%%%%%%%%%%%%%%%%%%%%%
%%%
%%%  Rees algebra attached to F_s(X)
%%%  SACADO A OTRO FICHERO
%%%%%%%%%%%%%%%%%%%%%%%%%%%%%%%%%%%%%%%%%%%%%%%%%%%%%%%%%%%%
%%%%%%%%%%%%%%%%%%%%%%%%%%%%%%%%%%%%%%%%%%%%%%%%%%%%%%%%%%%%

\part{Strong transversality} \label{parte_3} 

\section{Characterization of strongly transversal morphisms} \label{demo_theo}

Let $k$ be a perfect field, and let $\beta : X' \to X$ be a   transversal morphism of singular varieties over $k$. Recall that this is a finite and dominant morphism such that $\max\mult(X') = r \cdot \max\mult(X)$, where $r$ is the generic rank of $\beta : X' \to X$ (see Definition~\ref{def:transversal}). According to Theorem~\ref{thm:rest_definida_b} and Definition~\ref{def:algebra_restringida_b}, one can attach an intrinsic Rees algebra to the stratum of maximum multiplicity of $X$, say $\G_X$. Recall that this algebra is defined in \'etale topology. Similarly, one can attach an intrinsic algebra to the stratum of maximum multiplicity of $X'$, say $\G_{X'}$ (see Remark \ref{contenidos_algebras} below). As Proposition~\ref{prop:GX-subset-GX'} shows, there a relation between $\G_X$ and $\G_{X'}$. Along this section we will study the connection between the transversal morphism $\beta$ and the algebras $\G_X$ and $\G_{X'}$. The main result is Theorem~\ref{thm:strong-homeo}.

\begin{proposition}[{\cite[Proposition~6.3]{Carlos}}]
\label{prop:GX-subset-GX'}
Let $\beta : X' \to X$ be a finite morphism of singular varieties defined over a perfect field $k$. Let $\G_X \subset \mathcal{O}_X[W]$ and $\G_{X'} \subset \mathcal{O}_{X'}[W]$ denote the intrinsic Rees algebras attached to the strata of maximum multiplicity of $X$ and $X'$ respectively. If $\beta$ is transversal, then there is an inclusion $\G_X \subset \G_{X'}$.
\end{proposition}

\begin{theorem} \label{thm:strong-homeo}
Let $\beta : X' \to X$ be a transversal morphism of generic rank $r$ between two  singular  varieties  defined over a perfect  field $k$. Then:
\begin{enumerate}
\item If $\beta : X' \to X$ is strongly transversal then the inclusion $\G_X\subset \G_{X'}$ is finite; 
%\item The converse holds when $k$ is of characteristic zero.
\item If $k$ is a field of characteristic zero, then the converse holds. Namely, if the inclusion $\G_X \subset \G_{X'}$ is finite, then $\beta : X' \to X$ is strongly transversal.
\end{enumerate}
\end{theorem}

The proof of this theorem requires some preliminary technical results. The main ideas of the proof can be found in Remark~\ref{rmk:proof:strong-homeo} below. The proof will be addressed in \ref{demo}.

\begin{remark}\label{contenidos_algebras}
Recall that $\G_X$ and $\G_{X'}$ are only defined locally in \'etale topology. However, as we will explain below, given a point $\xi \in F_s(X)$, it is possible to find a suitable \'etale neighborhood of $X$ at $\xi$, say $\widetilde{X} \to X$, so that one can construct the intrinsic algebra $\G_{\widetilde{X}}$ associated to $\widetilde{X}$, as well as the intrinsic algebra $\G_{\widetilde{X}'}$ associated to $\widetilde{X}' = X' \times_X \widetilde{X}$ (see Remark~\ref{rmk:proof:strong-homeo}). It is in this setting that there is an inclusion $\G_{\widetilde{X}'} \subset \G_{\widetilde{X}}$, and in which we will compare these algebras.
\end{remark}

\begin{remark} \label{rmk:thm:strong-homeo}
In case that $\characteristic k = 0$, the theorem says that $\beta : X' \to X$ is strongly transversal if and only if the inclusion $\G_X\subset \G_{X'}$ is finite. That is, we can characterize the strong transversality of $\beta$ by means of the intrinsic algebras $\G_X$ and $\G_{X'}$. On the other hand, when $\characteristic k > 0$, property (1) holds as well,    but the condition on the characteristic is necessary in (2),    as the following example shows. 
\end{remark}

\begin{example} \label{contraejemplo}
Here we exhibit a transversal morphism $\beta : X' \to X$ where $\G_{X'}$ is finite over $\G_X$, but $X'$ is not strongly transversal to $X$.

Let $k$ be a field of characteristic $2$. Consider the varieties $X = \Spec(B)$ and $X' = \Spec(B')$, where
\[
	B = k[t,x,y] / \langle y^4 - x^{13} \rangle
	\quad \text{and} \quad
	B' = k[t,x,y,z] / \langle z^2 - x^5, y^4 - x^{13} \rangle.
\]%
Note that there is a natural finite and dominant morphism $\beta : X' \to X$ given by the inclusion $B \subset B'$, which has generic rank $2$. One can check that $\max\mult(X) = 4$ and $\max\mult(X') = 8$ (both values are attained at the origin), and thus $\beta : X' \to X$ is transversal (see Definition~\ref{def:transversal}). Next we shall show that $\G_{X'}$ is finite over $\G_{X}$, but $X'$ is not strongly transversal to $X$.

Set $S = k[t,x]$, $V = \Spec(k[t,x,y])$, and $V'=\Spec(k[t,x,y,z])$. Since $S  \subset B$ is a finite extension of generic rank $4$, the closed set $F_4(X)$ is represented in $V$ by the differential Rees algebra
\[
	\G = k[t,x,y][x^{12} W^3, (y^4-x^{13})W^4]
	\subset k[t,x,y][W],
\]%
and $\G_X = \overline{\G \rvert_X}$ (see \ref{Main_Ideas_Mult} and Definition~\ref{def:algebra_restringida_b}). Similarly, the differential algebra
\[
	\G' = k[t,x,y,z][x^4 W, (z^2-x^5) W^2, (y^4 - x^{13}) W^4]
	\subset k[t,x,y,z][W]
\]%
represents $F_8(X')$ in $V'$, and $\G_{X'} = \overline{\G' \rvert_{X'}}$. Note that the map $\G \to \G \rvert_X$ sends the element $(t^4-x^7)W^4$ to zero. Hence $\G \rvert_X$ is generated by the class of $x^{12} W^3$ in $B[W]$. Similarly, one readily checks that $\G' \rvert_{X'}$ is generated by the class of $x^4 W$ in $B'[W]$. Since $(x^4 W)^3 = x^{12} W^3$, it follows that $\G' \rvert_{X'}$ is integral over $\G \rvert_X$. Thus we see that $\G_{X'} = \overline{\G' \rvert_{X'}}$ is finite over $\G_X = \overline{\G \rvert_X}$.

Next we will show that $\beta : X' \to X$ is not strongly transversal. Consider the blow ups of $X' \subset V'$ and $X \subset V$ at the origin, i.e., at the closed points defined by $\langle t,x,y \rangle$ and $\langle t,x,y,z \rangle$ respectively. These blow ups induce a commutative diagram of inclusions and finite morphisms as follows:
\[ \xymatrix@R=15pt{
	X' \subset V' \ar[d] &
		X'_1 \subset V'_1 \ar[d] \ar[l] \\
	X \subset V &
		X_1 \subset V_1. \ar[l]
} \]%
One can see that the strict transforms of $X$ and $X'$ on the $t$-chart of these blow ups are given by
\[
	X_1 = \Spec \bigl( k[t_1,x_1,y_1] / \langle y_1^4 - t_1^9 x_1^{13} \rangle \bigr), \ \text{ and } \ 
	X'_1 = \Spec \bigl( k[t_1,x_1,y_1,z_1]
		/ \langle z_1^2 - t_1^3 x_1^5, y_1^4 - t_1^9 x_1^{13} \rangle \bigr)
\]%
(see \cite[Remark~5.3]{simplification-equimult}). Next, consider the centers $Y_1 \subset F_4(X_1)$ and $Y'_1 \subset F_8(X'_1)$ defined by $\langle t_1,y_1 \rangle$ and $\langle t_1,y_1,z_1 \rangle$ respectively. Note that $Y'_1$ sits on $Y_1$ via the finite morphism $X'_1 \to X$. The blow ups of $X_1 \subset V_1$ and $X'_1 \subset V'_1$ along these centers induce a commutative diagram
\[ \xymatrix@R=15pt{
	X' \subset V' \ar[d] &
		X'_1 \subset V'_1 \ar[d] \ar[l] &
		X'_2 \subset V'_2 \ar[d] \ar[l] \\
	X \subset V &
		X_1 \subset V_1 \ar[l] &
		X_2 \subset V_2. \ar[l]
} \]%
Moreover, on the $t_1$-chart, $X_2$ and $X'_2$ are given by
\[
	\Spec \bigl( k[t_2,x_2,y_2] / \langle y_2^4 - t_2^5 x_2^{13} \rangle \bigr), \ \text{ and  } \  
	\Spec \bigl( k[t_2,x_2,y_2,z_2]
		/ \langle z_2^2 - t_2 x_2^5, y_2^4 - t_2^5 x_2^{13} \rangle \bigr)
\]%
respectively. Next observe that the (non-closed) point $\xi_2 \in X_2$ defined by the ideal $\langle t_2, y_2 \rangle$ belongs to $F_4(X_2)$. However, the (non-closed) point $\xi'_2 \in X'_2$ defined by $\langle t_2,y_2,z_2\rangle$, which sits on $\xi_2$, does not belong to $F_8(X'_2)$, since $x_2$ is unit in $k[t_2,x_2,y_2,z_2]_{\langle t_2,y_2,z_2\rangle}$. This shows that $F_8(X'_2)$ is not homeomorphic to $F_4(X_2)$, and therefore $X'$ is not strongly transversal to $X$.
\end{example}

\begin{remark} \label{rmk:proof:strong-homeo} 
Here we outline the main ideas of the proof of Theorem~\ref{thm:strong-homeo} (the proof itself is addressed in \ref{demo}). Suppose that $X$ and $X'$ were affine, say $X = \Spec(B)$ and $X' = \Spec(B')$, with the morphism $\beta : X' \to X$ given by a finite inclusion $B \subset B'$. After replacing $B$ and $B'$ by suitable \'etale neighborhoods   (which can be done without loss of generality by Remarks \ref{rmk:strongly-transversal-equiv-surjective} and  \ref{rmk:multiplicity-etale}), we may assume that there is a regular $k$-algebra contained in $B$, say $S$, so that $S \subset B$ is a finite inclusion of generic rank $s = \max\mult(X)$, and $S \subset B'$ is a finite inclusion of generic rank $rs = \max\mult(X')$ (see the discussion in \ref{Main_Ideas_Mult}). In other words, the finite morphisms $X \to \Spec(S)$ and $X' \to \Spec(S)$ are transversal.\color{black}

Next we proceed as in \ref{Main_Ideas_Mult}. Set $d = \dim B = \dim B'$. One can construct a closed immersion of $X$ in a regular ambient space $V^{(d+n)} = \Spec(S[Z_1, \dotsc, Z_n])$, say $X \hookrightarrow V^{(d+n)}$, and a differential Rees algebra over $V^{(d+n)}$, say $\G^{(d+n)}$, which represents $F_s(X)$ (see diagram \eqref{diag:rmk75} below).
Similarly, there is an integer $n' \geq n$ so that one can construct an embedding of $X'$ in $V^{(d+n')} = \Spec(S[Z_1, \dotsc, Z_n, \dotsc, Z_{n'}])$, and a differential Rees algebra over $V^{(d+n')}$, say $\G^{(d+n')}$, which represents $F_{rs}(X')$. Thus we get a commutative diagram
\begin{equation}  \label{diag:rmk75}
\xymatrix@R=15pt @C=0em {
	\G^{(d+n')} &
		V^{(d+n')} \ar[d] &
		\hskip2em &
		X' \ar[d]^{\beta} \ar@{_(->}[ll] \\
	\G^{(d+n)} &
		V^{(d+n)} \ar[d] &
		&
		X \ar[lld] \ar@{_(->}[ll] \\
	&
		\Spec(S) .
} \end{equation}%
Recall that, in this setting, 
\[
	\G_X = \overline{ \left ( \G^{(d+n)}_{|_{X}} \right ) } \subset B[W] ,
	\quad \text{and} \quad
	\G_{X'} = \overline{ \left ( \G^{(d+n')}_{|_{X'}} \right ) } \subset B'[W].
\]%
In addition, since $X \to \Spec(S)$ is transversal, one readily checks that $V^{(d+n)} \to \Spec(S)$ is $\G^{(d+n)}$-admissible. Similarly, one can check that the morphism $V^{(d+n')} \to \Spec(S)$ is $\G^{(d+n')}$-admissible. Thus, by Corollary~\ref{crl:elim-alg-without-closure}, the Rees algebras
\[
	\mathcal{H} = \bigl( \G^{(d+n)} \cap S[W] \bigr) \subset S[W]
	\quad \text{and} \quad
	\mathcal{H}' = \bigl( \G^{(d+n')} \cap S[W] \bigr) \subset S[W]
\]%
are two elimination algebras of $\G^{(d+n)}$ and $\G^{(d+n')}$ over $S$ respectively.

Observe that $\mathcal{H} \subset S[W]$ and $\G_X \subset B[W]$, and that $S[W] \subset B[W]$ is a finite extension of Rees algebras because $B$ is finite over $S$. In the previous setting, we will also show that there is a commutative diagram of inclusions of Rees algebras, say
\[ \xymatrix@R=15pt{
	\G_X \ar@{^(->}[r] &
		\G_{X'} \\
	\mathcal{H} \ar@{^(->}[r] \ar@{^(->}[u] &
		\mathcal{H}', \ar@{^(->}[u]
} \]%
where both vertical arrows are finite (see Theorem~\ref{enteros_cociente} and Corollary~\ref{crl:extension_theorem} below). Thus $\G_X \subset \G_{X'}$ is finite if and only if $\mathcal{H} \subset \mathcal{H}'$ is so. These ideas will also be used in Section~\ref{construccion} to prove Theorem~\ref{rango_r}.
 
Note that, so far, all arguments are characteristic free. In addition, when the characteristic is zero, it turns out that $\mathcal{H}$ and $\mathcal{H}'$ are strongly linked to $\G^{(d+n)}$ and $\G^{(d+n')}$ respectively (see Remark~\ref{rmk:elim-alg-characteristic}~(2)). Therefore, in the case of characteristic zero, since $\G^{(d+n)}$ represents $F_s(X)$ and $\G^{(d+n')}$ represents $F_{rs}(X')$, it follows that $\beta : X' \to X$ is strongly transversal if and only if $\mathcal{H}$ and $\mathcal{H}'$ are weakly equivalent over $\Spec(S)$.
Moreover, since $\mathcal{H}$ and $\mathcal{H}'$ are differential and $\mathcal{H} \subset \mathcal{H}'$, to prove this latter statement it suffices to check that $\mathcal{H}'$ is integral over $\mathcal{H}$ or, equivalently, that $\overline{\mathcal{H}} = \overline{\mathcal{H}'}$ (see Theorem~\ref{thm:canonical-indiana}). In summary, when the characteristic is zero we have that
\begin{multline*}
 	\G_{X} \subset \G_{X'} \text{ is integral}
 	\Longleftrightarrow
 		\mathcal{H} \subset \mathcal{H}' \text{ is integral}
 	\Longleftrightarrow
	 	\mathcal{H} \text{ and } \mathcal{H}' \text{ are weakly equivalent}
	 	\Longleftrightarrow \\
	\Longleftrightarrow
 		F_s(X) \text{ and } F_{rs}(X') \text{ are strongly homeomorphic}
 		%F_{rs}(X') \text{ is strongly homeomorphic to } F_s(X)
 	\Longleftrightarrow
 		\beta : X' \to X \text{ is strongly transversal.}
\end{multline*}
\end{remark}

In order to prove Theorem~\ref{thm:strong-homeo} we need two preliminary results. The next theorem concerns the elimination of one variable, whereas the Corollary concerns the elimination of several variables. Finally, the proof of Theorem~\ref{thm:strong-homeo} will be given in \ref{demo}.

\begin{theorem}[{cf. \cite[Theorem 4.11]{hpositive}}] \label{enteros_cociente}
Let $k$ be a perfect field, let $S$ be a $d$-dimensional smooth $k$-algebra, and  let $\G^{(d+1)} \subset S[Z][W]$ be a differential Rees algebra over $S[Z]$. Suppose that there is a  monic polynomial of degree $n$, say $f(Z) \in S[Z]$, such that $f(Z)W^n \in \G^{(d+1)}$. Set $B = S[Z] / \langle f(Z) \rangle$. Then the smooth morphism $\beta : \Spec(S[Z]) \to \Spec(S)$ is $\G^{(d+1)}$-admissible, and
\[
	\G^{(d)} = \G^{(d+1)} \cap S[W]
\]%
is an elimination algebra of $\G^{(d+1)}$. Moreover, there is an inclusion of Rees algebras, say $\G^{(d)} \subset \G^{(d+1)}_{|_{B}}$, which is a finite extension of graded algebras.
\end{theorem}

\begin{corollary} \label{crl:extension_theorem}
Let $k$ be a perfect field, let $S$ be a smooth $k$-algebra of dimension $d$. Let $Z_1, \dotsc, Z_h$ denote variables and, for $i = 1, \dotsc, h$, let $f_i(Z_i) \in S[Z_i]$ be a monic polynomial of degree $l_i$, and set
\[
	B
	:= S[Z_1,\cdots,Z_h] / \langle f_1(Z_1), \cdots, f_h(Z_h) \rangle .
\]%
Let $\G^{(d+h)}$ be a differential Rees algebra over $S[Z_1,\ldots,Z_h]$ containing $f_1(Z_1)W^{l_1}, \ldots, f_h(Z_h)W^{l_h}$. Then the natural morphism $\text{Spec}(S[Z_1\ldots,,Z_h])\to  \text{Spec}(S)$ is ${\mathcal G}^{(d+h)}$-admissible and  
\[
	\G^{(d)} := \G^{(d+h)} \cap S[W].
\]%
is an elimination algebra of  $\G^{(d+h)}$. Furthermore, there is an inclusion of Rees algebras, say
\begin{equation} \label{eq:crl:extension_theorem_claim_1}
	\G^{(d)}
	\subset \G_{|_B}^{(d+h)} ,
\end{equation}
that is finite. Moreover, as a consequence, there is another inclusion of Rees algebras over $S$, say
\begin{equation} \label{eq:crl:extension_theorem_claim_2}
	\G^{(d)}
	\subset \left (\G^{(d+h)}_{|_B} \cap S[W] \right ),
\end{equation}%
which is also finite.
\end{corollary}

\begin{proof}
The result follows from the following observation by an inductive argument. Suppose that $h=2$. Recall that $\G^{(d+2)}$ is a differential algebra containing $f_1(Z_1) W^{l_1}$ and $f_2(Z_2) W^{l_2}$. Then one can check that $\tau_{{\mathcal G^{(d+2)}}}\geq 2$ at all $\xi\in \Sing {\mathcal G}^{(d+2)}$ and that $\text{Spec} (S[Z_1,Z_2])\to \text{Spec} (S[Z_1])$ is ${\mathcal G}^{(d+2)}$-admissible (see \cite[\S8]{BrV}). Set $\G^{(d+1)} = \G^{(d+2)} \cap S[Z_1][W]$, and consider the diagram
$$\xymatrix@R=15pt{ \G^{(d+2)} & S[Z_1,Z_2] \ar[r] &    B_2= S[Z_1][Z_2]/\langle f_2(Z_2)\rangle \ar[r] & B=S[Z_1,Z_2]/\langle f_1(Z_1), f_2(Z_2)\rangle \\ 
	 \G^{(d+1)} & S[Z_1] \ar[u] \ar[ur] \ar[r] & B_1=S[Z_1]/\langle f_1(Z_1)\rangle \ar[ur] &  \\
	\G^{(d)} & S \ar[u]  \ar[ur] & & }$$
Recall that, by Corollary~\ref{crl:elim-alg-without-closure}, $\G^{(d+1)}$ is differential, and it is an elimination algebra of $\G^{(d+2)}$ over $S[Z_1]$. Similarly, $\G^{(d)}=\G^{(d+1)}\cap S$ is also an elimination algebra of $\G^{(d+1)}$ over $S$ (again, it can be checked that $\text{Spec} (S[Z_1])\to \text{Spec} (S)$ is ${\mathcal G}^{(d+1)}$-admissible and  Corollary~\ref{crl:elim-alg-without-closure} applies). Now \eqref{eq:crl:extension_theorem_claim_1} follows from Theorem \ref{enteros_cociente}, as we have that:
\begin{itemize}
\item There is an inclusion $\G^{(d+1)} \subset \G^{(d+2)}_{|_{B_2}}$ that is finite, hence  $\G^{(d+1)}_{|_{B_1}} \subset \G^{(d+2)}_{|_{B}}$ is also finite; 

\item There is an inclusion, 
$\G^{(d)}\subset \G_{|_{B_1}}^{(d+1)}$ 
that is finite.
\end{itemize}%
To conclude, observe that, since the extension $S \subset B$ is finite, all the following extensions are finite:
$$ \G^{(d)}
	\subset \left ( \G^{(d)}_{|_{B}} \cap S[W] \right )
	\subset \left ( \G^{(d+1)}_{|_{B}} \cap S[W] \right )
	\subset \left (\G^{(d+2)}_{|_{B}} \cap S[W] \right ) .$$
This proves \eqref{eq:crl:extension_theorem_claim_2}.
\end{proof}

\begin{parrafo} \label{demo}
\begin{proof}[\bf Proof of Theorem~\ref{thm:strong-homeo}]

It suffices to prove the theorem in the affine case. Hence we may assume that $X = \Spec(B)$, $X' = \Spec(B')$, and that $\beta : X' \to X$ is the morphism induced by  a finite inclusion $B \subset B'$ (see   Remark~\ref{rmk:proof:strong-homeo}). Set $d = \dim B = \dim B'$.  

After replacing $B$ by some \'etale extension  we may assume that there is a finite inclusion $S \subset B$ with $S$ smooth as in \ref{Main_Ideas_Mult}. This \'etale extension of $B$ induces an \'etale extension of $B'$ which we will denote by $B'$ again, and the induced finite inclusion $S \subset B'$ is in the same setting as that of \ref{Main_Ideas_Mult}. Observe, that, by the hypotheses, $F_{rs}(B')$ is homeomorphic to its image by $\beta$, which necessarily sits inside $F_s(B)$. Then there is a commutative diagram of 
inclusions and finite morphisms, 
$$
\xymatrix@R=15pt{ \G^{(d+n')} & C'=S[Z_1,\ldots,Z_n,Z_{n+1},\ldots, Z_{n'}] \ar[r] & B' = S[\theta_1,\ldots, \theta_n, 
\theta_{n+1},\ldots,\theta_{n'}] \\
 \G^{(d+n)} & C=S[Z_1,\ldots,Z_n] \ar[u] \ar[r] & B=S[\theta_1,\ldots, \theta_n]  \ar[u]   \\
 &   S \ar[ur] \ar[u] }
$$
where each monic polynomial $f_i(Z_i)\in S[Z_i]$ of degree $l_i$ is the minimum polynomial of $\theta_i$ over $S$, for $i = 1, \ldots, n, \ldots, n'$, as in \ref{Main_Ideas_Mult}. Let $\G^{(d+n)}$ be the differential $C$-Rees algebra generated by $f_1(Z_1) W^{l_1}, \dotsc, f_n(Z_n) W^{l_n}$, and $\G^{(d+n')}$ the differential $C'$-Rees algebra generated by
\[
	f_1(Z_1)W^{l_1}, \dotsc, f_n(Z_n)W^{l_n},
	f_{n+1}(Z_{n+1})W^{l_{n+1}} \dotsc,f_{n'}(Z_{n'})W^{l_{n'}} .
\]%
Then $\G^{(d+n)}$ represents $F_s(B)$ in $\Spec(C)$, and $\G^{(d+n')}$ represents $F_{rs}(B')$ in $\text{Spec}(C')$. By definition, $\G_B$ is the integral closure in $B[W]$ of the restriction of $\G^{(d+n)}$ to $B$, say $\G^{(d+n)}_{|_B}$, while $\G_{B'}$ is the integral closure in $B'$  of the restriction of $\G^{(d+n')}$  to $B'$, say  $\G^{(d+n')}_{|_{B'}}$ (see Definition~\ref{def:algebra_restringida_b}).

Next we prove (1). Suppose that $\beta$ is strongly transversal. In such case $\G^{(d+n')}$ is strongly linked to $\G^{(d+n)}$, and hence $\G^{(d+n)}$ is an elimination algebra of $\G^{(d+n')}$ (see Remark~\ref{rmk:elim-alg-characteristic}~(1)). Set
$$ D = S[Z_1,\ldots,Z_n,Z_{n+1},\ldots, Z_{n'}] / \langle f_{n+1}(Z_{n+1}), \ldots, f_{n'}(Z_{n'}) \rangle.$$   
The algebra $D$ fits into the previous diagram as follows: 
\[
\xymatrix @R=15pt{
	\G^{(d+n')} & C' \ar[r] & D \ar[r] & B' \\
	\G^{(d+n)} & C \ar[u] \ar[rr] \ar[ru] & \hskip6em & B \ar[u] \\
	%& S \ar[urr] \ar[u]
} \]%
By Corollary \ref{crl:extension_theorem} there is an inclusion $\G^{(d+n)} \subset \G^{(d+n')}_{|_{D}}$  which is finite. Therefore, the inclusion $\G^{(d+n)} \subset \G^{(d+n')}_{|_{B'}}$ is finite, as $B'$ is a quotient of $D$. Thus (1) follows because $\G_B$ is the integral closure of $ \G^{(d+n)}_{|_{B}}$ in $B[W]$, and $\G_{B'}$ is the integral closure of $ \G^{(d+n')}_{|_{B'}}$ in $B'[W]$.

To prove (2) we will restrict to characteristic zero. Assume that $\G^{(d+n)}_{|_{B}} \subset \G^{(d+n')}_{|_{B'}}$ is a finite extension of Rees algebras. Set
\[
	\mathcal{H} = \G^{(d+n)} \cap S[W] ,
	\quad \text{and} \quad
	\mathcal{H}' = \G^{(d+n')} \cap S[W].
\]%
By Corollary~\ref{crl:elim-alg-without-closure}, $\mathcal{H}$ and $\mathcal{H}'$ are two elimination algebras of $\G^{(d+n)}$ and $\G^{(d+n')}$ over $S$ respectively. Since we are assuming characteristic zero, Remark~\ref{rmk:elim-alg-characteristic}~(2) says that $\mathcal{H}$ is strongly linked to $\G^{(d+n)}$, and $\mathcal{H}'$ is strongly linked to $\G^{(d+n')}$. Thus, in order to prove that $\beta : X' \to X$ is strongly transversal, it suffices to show that $\mathcal{H}$ and $\mathcal{H}'$ are weakly equivalent (see Remark~\ref{rmk:proof:strong-homeo}). To this end, observe that we have a commutative diagram of inclusions of Rees algebras as follows:
\[ \xymatrix@R=15pt{
	\G^{(d+n)}_{|_{B}} \ar@{^{(}->}[r] &
		\G^{(d+n')}_{|_{B'}} \\
	\mathcal{H} \ar@{^{(}->}[r] \ar@{^{(}->}[u] &
		\mathcal{H}'. \ar@{^{(}->}[u]
} \]%
The vertical arrows are finite by Corollary~\ref{crl:extension_theorem}, and the top horizontal arrow is finite by hypothesis. Hence $\G^{(d+n')}_{|_{B'}}$ is finite over $\mathcal{H}$, and therefore $\mathcal{H}'$ is also finite over $\mathcal{H}$. This latter condition shows that $\mathcal{H}$ and $\mathcal{H}'$ are weakly equivalent and, as a consequence, that $\beta : X' \to X$ is strongly transversal.
\end{proof}
\end{parrafo}

\begin{remark} \label{equivalent_resolutions}
Assume that the characteristic is zero, so that one can apply an algorithm of resolution of Rees algebras (see Definition~\ref{local_V}). A feature of algorithmic resolution of algebras is that it does not distinguish between weakly equivalent algebras (see Definition~\ref{def_weak_equiv}). Namely, if $\mathcal{H}$ and $\mathcal{H}'$ are two weakly equivalent algebras over a regular variety $V^{(d)}$ of characteristic zero, then the algorithm produces the same resolution for $\mathcal{H}$ and $\mathcal{H}'$.

Consider a strongly transversal morphism of singular varieties of characteristic zero, say  $\beta : X' \to X$, with $s = \max\mult(X)$ and $rs = \max\mult(X')$ (see Definition~\ref{def:strongly-transversal}). Fix the setting as in Remark~\ref{rmk:proof:strong-homeo}. Since the characteristic is zero, the simplification of $F_s(X)$ and $F_{rs}(X')$ can be reduced to the resolution of two Rees algebras, say $\mathcal{H}$ and $\mathcal{H}'$ respectively, over a regular ambient space $\Spec(S)$. Since $\beta : X' \to X$ is strongly transversal, $\mathcal{H}$ and $\mathcal{H}'$ turn out to be weakly equivalent. Thus one can see that the algorithmic resolution of $F_s(X)$ induces that of $F_{rs}(X')$, and vice versa.
\end{remark}

%%%%%%%%%%%%%%%%%%%%%%%%%%%%%%%%%%%%%%%%%%%%%%%%%%%%%%%%%%%%
%%%%%%%%%%%%%%%%%%%%%%%%%%%%%%%%%%%%%%%%%%%%%%%%%%%%%%%%%%%%
%%%
%%% SECTION 7 REESCRITURA
%%%
%%% On the construction of strongly transversal morphisms
%%%
%%%%%%%%%%%%%%%%%%%%%%%%%%%%%%%%%%%%%%%%%%%%%%%%%%%%%%%%%%%%
%%%%%%%%%%%%%%%%%%%%%%%%%%%%%%%%%%%%%%%%%%%%%%%%%%%%%%%%%%%%

\section{On the construction of strongly transversal morphisms}\label{construccion}

In the previous sections we have studied conditions under which a finite morphism of singular varieties, say $\beta : X' \to X$, is strongly transversal. Along this section we proceed the other way around. Namely, we start with a singular variety $X$ with field of fractions $K$ and, given a finite field extension $L / K$, we ask whether there exists a singular variety $X'$ with field of fractions $L$ endowed with a strongly transversal morphism $\beta : X' \to X$. We do not know how to construct such $X'$ in general. However, as we will explain in the following lines, we can achieve such result locally in \'etale topology. The main result of this section is Theorem~\ref{rango_r}.

%\begin{remark}
%The behavior of the multiplicity is naturally compatible with \'etale topology. Namely, given an equidimensional variety $X$ and an \'etale morphism $\varepsilon : \widetilde{X} \to X$, one has that $\mult(\tilde{\xi}) = \mult\bigl(\varepsilon(\tilde{\xi})\bigr)$ for every $\tilde{\xi} \in \widetilde{X}$. Moreover, every sequence of blow-ups along regular equimultiple centers over $X$, say
%\[ \xymatrix {
% 	X &
% 		X_1 \ar[l] &
% 		\dotsb \ar[l] &
% 		X_l, \ar[l]
%} \]%
%induces by base change a sequence blow-ups along regular equimultiple centers over $\widetilde{X}$, say
%\[ \xymatrix {
% 	\widetilde{X} &
% 		\widetilde{X}_1 \ar[l] &
% 		\dotsb \ar[l] &
% 		\widetilde{X}_l, \ar[l]
%} \]%
%and a commutative diagram,
%\[ \xymatrix @R=15pt {
% 	\widetilde{X} \ar[d]^{\varepsilon} &
% 		\widetilde{X}_1 \ar[l] \ar[d]^{\varepsilon_1} &
% 		\dotsb \ar[l] &
% 		\widetilde{X}_l \ar[l] \ar[d]^{\varepsilon_l} \\
% 	X &
% 		X_1 \ar[l] &
% 		\dotsb \ar[l] &
% 		X_l, \ar[l]
%} \]%
%where each $\varepsilon_i$ is \'etale.
%\end{remark}

\begin{remark}
Let $X$ be a normal variety with field of fractions $K$ and maximum multiplicity $s \geq 2$. For a fixed point $\xi \in F_s(X)$, consider an \'etale neighborhood of $X$ at $\xi$, say $\widetilde{X} \to X$. Note that $\widetilde{X}$ might not be irreducible. However, as $X$ is normal, $\widetilde{X}$ should be a disjoint union of irreducible components, each one being normal (see \cite[Proposition~VII.2.2, p.~75]{Raynaud}). Thus, after replacing $\widetilde{X}$ by one of its irreducible components, we may assume that $\widetilde{X}$ is irreducible. Note also that, if $L$ is a finite extension of $K$ of rank $r$ and $\widetilde{K}$ denotes the field of fractions of $\widetilde{X}$, then $\widetilde{L} = L \otimes_K \widetilde{K}$ is a finite extension of $\widetilde{K}$ of generic rank $r$.
\end{remark}

%Let $X$ be a singular variety with maximum multiplicity $s \geq 2$ over a perfect field $k$. Let $K$ denote the field of fractions of $X$, and let $L$ be a finite field extension of $K$ of rank $r$. For a fixed point $\xi \in F_s(X)$, consider an \'etale neighborhood of $X$ at $\xi$, say $\widetilde{X} \to X$. Let $\widetilde{K}$ denote the quotient field of $\widetilde{X}$. Note that, in general, $\widetilde{K}$ might not be a field. However, as $\widetilde{X} \to X$ is \'etale, $\widetilde{K}$ is \'etale over $K$, and hence it is reduced. Thus $\widetilde{K}$ can be seen as a direct sum of separable extensions of $K$. Note also that, as $L$ is free of rank $r$ over $K$, the ring $\widetilde{L} = L \otimes_K \widetilde{K}$ is free of rank $r$ over $\widetilde{K}$. The following theorem proves that, for a suitable choice of $\widetilde{X}$, one can construct a scheme of finite type over $k$, say $\widetilde{X}'$, with field of fractions $\widetilde{L}$ and endowed with a strongly transversal morphism $\beta : \widetilde{X}' \to \widetilde{X}$ (in the sense of Remark~\ref{rmk:def-transversal-equidimensional}).

\begin{theorem} \label{rango_r}
Let $X$ be a singular variety over a perfect field $k$ with maximum multiplicity $s \geq 2$. Suppose that $X$ is normal with field of fractions $K$ and let $L$ be a finite field extension of $K$ of rank $r$. Then, for every point of  $\xi \in F_s(X)$, there exist an \'etale neighborhood of $X$ at $\xi$, say $\widetilde{X} \to X$, together with a scheme $\widetilde{X}'$ and a strongly transversal morphism $\beta : \widetilde{X}' \to \widetilde{X}$ of generic rank $r$ such that $\widetilde{L} = L \otimes_K \widetilde{K}$, where $\widetilde{K}$ and $\widetilde{L}$ denote the total quotient rings of $\widetilde{X}$ and $\widetilde{X}'$ respectively.
\end{theorem}

\begin{remark}\label{justificacion}  
The condition of normality on $X$ warranties that the irreducible components of an \'etale neighborhood of $X$ are disjoint. In general, an \'etale neighborhood of $X$, say $\widetilde{X}$, may not be irreducible. In such case, it is not possible to construct a strongly transversal morphism $\beta : \widetilde{X}' \to \widetilde{X}$, as transversal and strongly transversal morphisms have been defined just for the case in which $\widetilde{X}$ is a variety.

It is possible to extend the notions of transversal morphism and strongly transversal morphism to include the case in which $\widetilde{X}$ is an equidimensional scheme of finite type over a perfect field $k$. In this way, one could also generalize the previous theorem to the case in which $X$ is non-necessarily normal. These generalizations would not involve new ideas. However, they would require to introduce more notation and the the proof of Theorem~\ref{rango_r} would be more tricky. Thus, for the shake of simplicity and readability, we have chosen to keep the definitions of transversal morphism and strongly transversal morphism and Theorem~\ref{rango_r} as they are.

Nevertheless note that, from the point of view of resolution of singularities, there is no loss of generality on replacing a variety $X$ by its normalization, say $\overline{X}$, as any resolution of singularities of $X$ should factor through $\overline{X}$.
\end{remark}

\begin{proof}
After replacing $X$ by a suitable \'etale neighborhood, assume that $X$ is an affine variety, say $X = \Spec(B)$, whose intrinsic algebra is given by $\G_X = \G_B = \bigoplus_i J_i W^i \subset B[W]$. Let $\overline{B}$ denote the integral closure of the domain $B$ in $L$, and consider the integral closure of $\G_B$ in $\overline{B}[W]$, say
\[
	\overline{\G_B}
	= \overline{B} \oplus \widetilde{J}_1 W
		\oplus \widetilde{J}_2 W^2 \oplus \dotsb
	= \bigoplus_i \widetilde{J}_i W^i
	\subset \overline{B}[W].
\]%
Note that the integral closure of the extended ideal $J_i \overline{B}$ must be contained in the ideal $\widetilde{J}_i \subset \overline{B}$ but, in general, this inclusion is strict (see Remark~\ref{rmk:int-closure-Rees}).

As $\widetilde{J}_1$ is an ideal of $\overline{B}$ and $\overline{B}$ is finite  over $B$, the ideal $\widetilde{J}_1$ can also be regarded as a finite module over $B$. Fix a family of generators of  $\widetilde{J}_1$ regarded as a $B$-module, say $\widetilde{J}_1 = \langle \theta_1, \dotsc, \theta_m \rangle$. Set $B' = B[\theta_1, \dotsc, \theta_m] = B \bigl[ \widetilde{J}_1 \bigr] \subset L$. We claim that $X' = \Spec(B')$ is a singular variety with field of fractions $L$, and that the induced morphism $X' \to X$ is strongly transversal.

To prove the claim, first note that $\widetilde{J}_1$ is a non-zero ideal in $\overline{B}$, which is contained in $B'$. Hence $B'$ has $L$ as its field of fractions. On the other hand observe that, by construction, the elements $\theta_1 W, \dotsc, \theta_m W \in B'[W]$ are integral over $\G_B$. Then the strong transversality will follow from Proposition~\ref{prop:transversality-check} and Proposition~\ref{prop:strong-transversality-check} below.
\end{proof}

\begin{remark} \label{rmk:int-closure-Rees}
Let $B \subset B'$ be a finite extension of domains. Consider a Rees algebra over $B$, say
$\G = \bigoplus_{i \in \mathbb{N}} J_i W^i \subset B[W]$,
and let
$\overline{\G} = \bigoplus_{i \in \mathbb{N}} \widetilde{J}_i W^i \subset B'[W]$
denote the integral closure of $\G$ in $B'[W]$. 
Note that the integral closure of $J_i B'$ in $B'$ is contained in $\widetilde{J}_i$ for all $i$. However, in general, this inclusion is strict (see Example~\ref{ejemplo} below).
\end{remark}

\begin{example}\label{ejemplo}

\newcommand{\x}{\bar{x}}
\newcommand{\y}{\bar{y}}
\newcommand{\z}{\bar{z}}

Let $k$ be a field of characteristic zero, and consider the curve $X = \Spec(B)$ given by $B = k[x,y] / \langle y^3 + x^3 y + x^7 \rangle$. Let us denote by $\x$ and $\y$ the residue classes of $x$ and $y$ in $B$ respectively. Note that $X$ has maximum multiplicity $3$, and this value is reached at the origin. Following the notation of the proof, set $S = k[\x]$ and $B = S[\y]$. The maximum multiplicity locus of $X \subset \Spec(S[T])$ is represented by the Rees algebra
\[
	\G
	= S[T][ (T^3 + \x^3 T + \x^7) W^3 ]
	\subset S[T][W] ,
\]%
and a simple computation shows that $\Diff(\G) = S[T][TW, \x^2 W, \x^3 W^2]$. Therefore, the Rees algebra $\mathcal{H} = \Diff(\G) \cap S[W]$ is given by $\mathcal{H} = S[\x^2 W, \x^3 W^2] \subset S[W]$ and, in this case, $J_1 = \langle \x^2 \rangle \subset S$.

Next, consider the field extension $L = K[z] / \langle z^2 - x^3 \rangle$, where $K$ represents the field of fractions of $B$. Let $\overline{\mathcal{H}} = \bigoplus_l \tilde{J}_l W^l$ denote the integral closure of $\mathcal{H}$ in $L[W]$. Since
\[
	(\z W)^2 - \x^3 W^2 = (\z^2 - \x^3) W^2 = 0,
\]%
with $\x^3 W^2 \in J_2$, it follows that $\z W$ is integral over $\mathcal{H}$. That is, $\z \in \tilde{J}_1$. However, we claim that $\z$ is not integral over $J_1 \overline{S}$.
Indeed, according to \cite[Appendix~4, Theorem~1, p.~350]{ZS}, $\z$ is integral over $J_1 \overline{S}$ if and only if, for every valuation ring $R \subset L$ so that $S \subset R$, one has that $\z \in J_1 R$. Thus, in order to prove the claim it suffices to find a valuation ring $R \subset L$ so that $\z \notin J_1 R$. To this end, consider the subalgebra $k[\frac{\z}{\x}] \subset L$ (note that $k[\frac{\z}{\x}]$ is the ring obtained by blowing-up $k[\x,\z]$ along $\langle \x,\z \rangle$). Observe that $(\frac{\z}{\x})^2 = \x$ in $L$. Hence $S \subset k[\frac{\z}{\x}]$. Set $R_0 = k[\frac{\z}{\x}]_{\langle \frac{\z}{\x} \rangle}$, and let $L_0 \subset L$ denote its field of fractions. As $R_0$ is a regular Noetherian local ring of dimension $1$, it is a discrete valuation ring with parameter $\frac{\z}{\x}$. In this way, since $\z = \x \cdot \frac{\z}{\x} = (\frac{\z}{\x})^3$, one readily checks that
\begin{equation}  \label{eq:ultimo-ejemplo-J1R0}
	\z
	\notin J_1 R_0
	= \left\langle \x^2 \right\rangle R_0
	= \left\langle \frac{\z}{\x} \right\rangle^4 R_0.
\end{equation}%
By \cite[Theorem~VI.5, p.~12]{ZS}, $R_0$ can be extended to a valuation ring in $L$, say $R \subset L$, with parameter $\frac{\z}{\x}$. Thus \eqref{eq:ultimo-ejemplo-J1R0} implies that $\z \notin J_1 R$, and therefore $\z$ is not integral $J_1 \overline{S}$.
\end{example}

\begin{lemma} \label{lm:elim-GB}
Let $B$ be a singular domain of finite type over a perfect field $k$ with maximum multiplicity $s \geq 2$. Let $\G_B \subset B[W]$ denote the intrinsic algebra attached to $F_s(B)$. Consider a regular subalgebra of $B$, say $S$, so that the extension $S \subset B$ is finite of generic rank $s = \max\mult(B)$. That is, so that the induced morphism $\beta : \Spec(B) \to \Spec(S)$ is transversal. Then the $S$-Rees algebra $\mathcal{H} := \G_B \cap S[W]$ has the following properties:
\begin{enumerate}
\item The extension $\mathcal{H} \subset \G_B$ is finite.
\item The morphism $\beta$ maps $F_s(X)$ homeomorphically to $\Sing(\mathcal{H})$, i.e., $\beta(F_s(X)) \cong \Sing(\mathcal{H})$.
\end{enumerate}
\end{lemma}

\begin{proof}
Choose elements $\theta_1, \dotsc, \theta_n \in B$ such that $B = S[\theta_1, \dotsc, \theta_n]$. This presentation of $B$ induces a surjective morphism $S[T_1, \dotsc, T_n] \to B = S[\theta_1, \dotsc, \theta_n]$, where $T_1, \dotsc, T_n$ represent variables and $T_i \mapsto \theta_i$, and, in turn, it induces a closed immersion $\Spec(B) \hookrightarrow V^{(d+n)} = \Spec(S[T_1, \dotsc, T_n])$.
Let $K$ denote the field of fractions of $S$, and let $f_1(T_1), \dotsc, f_n(T_n)$ be the minimal polynomials of $\theta_1, \dotsc, \theta_n$ over $K$ respectively. As $S$ is regular, $f_i(T_i) \in S[T_i]$ (see \ref{Main_Ideas_Mult}). Set $N_i = \deg(f_i(T_i))$, and consider the Rees algebra
\[
	\G^{(d+n)}
	= \mathcal{O}_{V^{(d+n)}} \left [
		f_1(T_1) W^{N_1}, \dotsc, f_n(T_n) W^{N_n} \right ]
	\subset \mathcal{O}_{V^{(d+n)}}[W] .
\]%
Under these hypotheses, there is a commutative diagram, say
\begin{equation}  \label{diag:elim-GB-pre}
\xymatrix@R=15pt@C=0em {
	\G^{(d+n)} &
		V^{(d+n)} \ar[d]_-{\varphi} &
		\hskip2em &
		\Spec(B) \ar[lld]^-{\beta} \ar@{_(->}[ll] \\
	&
		\Spec(S) ,
} \end{equation}%
and it can be proved that $\G^{(d+n)}$ represents $F_s(X)$ in $V^{(d+n)}$ (see the discussion in \ref{Main_Ideas_Mult}).

Next, consider the Rees algebras
\[
	\mathcal{K}^{(d+n)}
	:= \overline{\Diff \left ( \G^{(d+n)} \right ) }
	\subset \mathcal{O}_{V^{(d+n)}}[W],
	\quad \text{ and } \quad
	\mathcal{K}^{(d)} := \mathcal{K}^{(d+n)} \cap S[W].
\]%
Recall that, by Definition~\ref{def:algebra_restringida_b},
\[
	\G_B
	= \overline{ \left ( \mathcal{K}^{(d+n)}_{|_{B}} \right ) }
	\subset B[W].
\]%
Thus there are inclusions $\mathcal{K}^{(d)} \subset \mathcal{H} \subset \G_B$. Moreover, Corollary~\ref{crl:extension_theorem} says that $\G_B$ is finite over $\mathcal{K}^{(d)}$, and hence $\mathcal{K}^{(d)} \subset \mathcal{H}$ is also a finite extension of subalgebras of $S[W]$. Since $\mathcal{K}^{(d+n)}$ is integrally closed,  $\mathcal{K}^{(d)}$ is integrally closed. Therefore $\mathcal{H} = \mathcal{K}^{(d)}$, and (1) follows from Corollary~\ref{crl:extension_theorem}.

On the other hand, since $\mathcal{K}^{(d+n)}$ is weakly equivalent to $\G^{(d+n)}$,  $\mathcal{K}^{(d+n)}$ represents $F_s(B)$ over $V^{(d+n)}$. As $\Spec(B) \to \Spec(S)$ is transversal, the commutativity of \eqref{diag:elim-GB-pre} implies that $\beta$ is a $\mathcal{K}^{(d+n))}$-admissible projection, and $\mathcal{H} = \mathcal{K}^{(d)}$ is an elimination algebra of $\mathcal{K}^{(d+n)}$ (see Definition~\ref{def:an-elimination-algebra}). Then  $\Sing(\mathcal{H}) \cong \Sing{(\mathcal{K}^{(d+n)})} = F_s(B)$ (see \ref{imagen_sing}), which proves (2).
\end{proof}

\begin{lemma} \label{lm:elim-GB-blowup}
Let $S \subset B$ be a finite extension of domains over a perfect field $k$ of generic rank $s = \max\mult(B)$, where $S$ is regular and $B$ is singular, and let $\mathcal{H} := \G_B \cap S[W]$ be as in Lemma~\ref{lm:elim-GB}. Consider a prime ideal $P \in F_s(B)$ such that $B / P$ is regular, and set $\mathfrak{p} = P \cap S \in \Sing(\mathcal{H})$. Recall that, in this setting, there is a natural commutative diagram as follows,
\[ \xymatrix@R=15pt @C=.5ex{
	& \Spec(B) \ar[d] & \quad &
		X_1 = \Bl_P(B) \ar[ll] \ar[d]  &  \\
	\mathcal{H}  &  \Spec(S) &&
		Z_1 = \Bl_{\mathfrak{p}}(S)\ar[ll] ,  & \mathcal{H}_1, 
} \]%
where the vertical arrows are finite morphisms, and $\mathcal{H}_1 \subset \mathcal{O}_{Z_1}[W]$ denotes the transform of $\mathcal{H}$ \textup{(}in the sense of \ref{weaktransforms}\textup{)}. Assume that $F_s(X_1) \neq \emptyset$  \textup{(}i.e., $\max\mult(X_1) = \max\mult(X) = s$\textup{)}, and let $\G_{X_1} \subset \mathcal{O}_{X_1}[W]$ denote the intrinsic algebra attached to $F_s(X_1)$. Then there is an inclusion of algebras $\mathcal{H}_1 \subset \G_{X_1}$ \textup{(}which in general is not finite\textup{)}.
\end{lemma}

\begin{proof}
Following the proof of Lemma~\ref{lm:elim-GB}, one can construct a commutative diagram as follows,
\[ \xymatrix @R=15pt @C=0em {
	\mathcal{K}^{(d+n)} &
		V^{(d+n)} \ar[d]_{\beta} &
		\hskip2em &
		\Spec(B) \ar[lld] \ar@{_(->}[ll] \\
	\mathcal{H} &
		\Spec(S) ,
} \]%
where $\mathcal{K}^{(d+n)}$ is a differential and integrally closed Rees algebra representing $F_s(B)$ in $V^{(d+n)}$, and $\mathcal{H} = \mathcal{K}^{(d+n)} \cap S[W] = \G_B \cap S[W]$ is an elimination algebra of $\mathcal{K}^{(d+n)}$ over $S$.

Note that $P$ defines a closed regular center contained in $F_s(B)$, say $Y$, which can also be regarded as regular center contained in $\Sing(\mathcal{K}^{(d+n)}) \subset V^{(d+n)}$. Moreover, $\beta(Y) \subset \Sing(\mathcal{H})$ is the regular center defined by the prime $\mathfrak{p}$ mentioned in the lemma. In this way, after blowing up $\Spec(B)$, $V^{(d+n)}$ and $\Spec(S)$ along these centers, we obtain a commutative diagram
\[ \xymatrix@R=1ex@C=0em{
	(V^{(d+n)}, \mathcal{K}^{(d+n)}) \ar[dddd] &
		&
		&
		(V^{(d+n)}_1,\mathcal{K}^{(d+n)}_1) \ar'[dd][dddd] \ar[lll] &
		\\
	\\
	&
		\Spec(B) \ar@{_(->}[uul] \ar[ddl] &
		\qquad &
		&
		X_1 = \Bl_P(B) \ar@{_(->}[uul] \ar[ddl] \ar[lll] \\
	\\
	(\Spec(S),\mathcal{H}) &
		&
		&
		(Z_1 = \Bl_{\mathfrak{p}}(S),\mathcal{H}_1), \ar[lll] 
} \]%
where $\mathcal{K}^{(d+n)}_1$ and $\mathcal{H}_1$ denote the transforms of $\mathcal{K}^{(d+n)}$ and $\mathcal{H}$ respectively. Moreover, locally at points of $\Sing(\mathcal{K}^{(d+n)}_1)$, there is an inclusion $\mathcal{O}_{Z_1} \subset \mathcal{O}_{V^{(d+n)}_1}$, and it can be checked that $\mathcal{H}_1 \subset \mathcal{K}^{(d+n)}_1$. In addition, if $F_s(X_1) \neq \emptyset$, then $\mathcal{K}^{(d+n)}_1$ represents $F_s(X_1)$ in $V_1^{(d+n)}$. Finally, $\mathcal{H}_1 \subset \G_{X_1}$ because  by Definition~\ref{def:algebra_restringida_b},  we have that
 $\G_{X_1}
	= \overline{\Diff(\mathcal{K}^{(d+n)}_1) \rvert_{X_1}}$. 
\end{proof}

\begin{proposition} \label{prop:transversality-check}
Let $B$ be a singular domain   of finite type   over a perfect field $k$ with maximum multiplicity $s > 1$, and let $\G_B \subset B[W]$ denote the intrinsic Rees algebra attached to $F_s(B)$ in the sense of Definition~\ref{def:algebra_restringida_b}. Let $B' = B[\theta_1, \dotsc, \theta_m]$ be a finite and dominant extension of $B$ of generic rank $r$. If $\theta_1 W, \dotsc, \theta_m W \in B'[W]$ are integral over $\G_B \subset B[W]$, then $\max\mult(B') = rs$,
%\[
%	\max\mult(B') = r \cdot \max\mult(B) = rs,
%\]%
and therefore the morphism $\beta : \Spec(B') \to \Spec(B)$ is transversal. Moreover, in such case, $\beta$ induces a natural homeomorphism between $F_{rs}(B')$ and $F_s(B)$.
\end{proposition}

\begin{proof}
As $B \subset B'$ has generic rank $r$, Zarisiki's formula (Theorem~\ref{MultForm}) says that $\max\mult(B') \leq r \cdot \max\mult(B) = rs$.
Thus, in order to check the transversality of $\beta$, we just need to show that $F_{rs}(B') \neq \emptyset$. To this end, we will argue as follows. Consider a prime ideal $P \in F_s(B)$. Since the extension $B \subset B'$ is finite and dominant, $\beta$ is surjective, and therefore there exists at least a prime $P' \subset B'$ sitting on $P$. We shall show that $P' \in F_{rs}(B')$. Note that this property implies that:
\begin{enumerate}
\item $F_{rs}(B') \neq \emptyset$, and hence $\beta$ is transversal;
\item As $\beta$ is transversal, $F_{rs}(B')$ is mapped homeomorphically onto $\beta(F_{rs}(B'))$, and $\beta(F_{rs}(B')) \subset F_s(B)$ (see Corollary~\ref{homeomorphism});
\item $F_{rs}(B')$ maps surjectively onto $F_s(B)$, as the previous argument holds for any $P \in F_s(B)$;
\item As a consequence of (2) and (3), the set $F_{rs}(B')$ is mapped homeomorphically to $F_s(B)$.
\end{enumerate}
Thus, in order to prove the lemma, it suffices to show that any prime $P \in F_s(B)$ is dominated by a prime $P' \in F_{rs}(B')$.

Fix a prime $P \in F_s(B)$, and let $P' \subset B'$ be a prime ideal sitting on $P$. After replacing $B$ by a suitable \'etale neighborhood, we may assume that there is a regular $k$-algebra contained in $B$, say $S$, so that $S \subset B$ is a finite extension of generic rank $s = \max\mult(B)$ (see \ref{Main_Ideas_Mult}). In other words, the induced map $\beta : \Spec(B) \to \Spec(S)$ is transversal. Consider the Rees algebra $\mathcal{H} := \G_B \cap S[W]$.
%\[
%	\mathcal{H} := \G_B \cap S[W].
%\]%
By Lemma~\ref{lm:elim-GB}, we have that $\G_B$ contains and is finite over $\mathcal{H}$. Set $\mathcal{H} = \bigoplus_{l} J_l W^l \subset S[W]$. As $\theta_1 W, \dotsc, \theta_m W \in B'[W]$ are integral over $\G_B$, and $\G_B$ is integral over $\mathcal{H}$, then $\theta_1 W, \dotsc, \theta_m W$ are integral over $\mathcal{H}$. Hence each $\theta_i$ must satisfy a relation of integral dependence of the form
\begin{equation} \label{eq:int-dep-rel}
	\theta_i^N + a_1 \theta_i^{N-1} + \dotsb + a_N = 0;
	\quad a_j \in J_j.
\end{equation}
%with $a_j \in J_j$.

Consider the prime ideal $\mathfrak{p} = P \cap S$. Since $P \in F_s(B)$ and $\Spec(B) \to \Spec(S)$ is a finite transversal morphism, $P$ is the unique prime of $\Spec(B)$ sitting on $\mathfrak{p}$, and hence $B_P = B \otimes_S S_{\mathfrak{p}}$. Thus we may assume without loss of generality that $S = S_\mathfrak{p}$, and $B = B_P$. By Lemma~\ref{lm:elim-GB}, we have that $\mathfrak{p} \in \Sing(\mathcal{H})$. Therefore, as each $a_j$ in \eqref{eq:int-dep-rel} belongs to $J_j$, we deduce that $\nu_{\mathfrak{p}}(a_j) \geq j$ for all $j$. In this way, \eqref{eq:int-dep-rel} can also be regarded as a relation of integral dependence of $\theta_i$ over the ideal $\mathfrak{p} B' \subset B'$. Thus, if a prime ideal $Q \subset B'$ contains $\mathfrak{p} B'$, then $\theta_1, \dotsc, \theta_m \in Q$. Note also that
\[
	B' / (PB' + \langle \theta_1, \dotsc, \theta_m \rangle)
	= B / P.
\]%
Since $B / P$ is a domain, we deduce that $PB' + \langle \theta_1, \dotsc, \theta_m \rangle$ is a prime ideal in $B'$, and hence $P' = PB' + \langle \theta_1, \dotsc, \theta_m \rangle$. As a consequence, we have that:
(i)  $P'$ is the unique prime of $B'$ dominating $P$; 
(ii) $P'$ is rational over $P$; and
(iii)  $P B'_{P'}$ is a reduction of $P' B'_{P'}$. In virtue of \ref{condicion_estrella}, this implies that $P' \in F_{rs}(P')$, which proves the lemma.
\end{proof}

\begin{proposition} \label{prop:strong-transversality-check}
Under the same hypotheses of Proposition~\ref{prop:transversality-check}, the morphism $\beta : \Spec(B') \to \Spec(B)$ is strongly transversal.
\end{proposition}

\begin{proof}
Proposition~\ref{prop:transversality-check} says that $\beta$ is transversal, and $\beta(F_{rs}(B')) = F_s(B)$. Thus we just need to check that the conditions of the lemma are preserved (locally) after blowing up.

Fix a prime ideal $P \in F_s(B)$ which defines a regular center in $\Spec(B)$, i.e., such that $B/P$ is regular, and let $P' \in F_{rs}(B')$ be the unique prime in $B'$ sitting on $P$. After replacing $B$ and $B'$ by suitable \'etale neighborhoods, we may assume that $B$ contains a regular subalgebra, say $S$, so that $S \subset B$ is a finite extension of generic rank $s = \max\mult(B)$. That is, so that the $\Spec(B) \to \Spec(S)$ is transversal. Then set $\mathfrak{p} = P \cap S$, and consider the Rees algebra $\mathcal{H} := \G_B \cap S[W]$. By Lemma~\ref{lm:elim-GB}, we have that $\mathfrak{p} \in \Sing(\mathcal{H})$, and hence it defines a permissible center for $\mathcal{H}$. Moreover, by Lemma~\ref{lm:blow-up-finite}, there are finite morphisms $\Bl_{P'}(B') \to \Bl_P(B) \to \Bl_{\mathfrak{p}}(S)$.

Fix a collection of generators of the ideal $\mathfrak{p} \subset S$, say $\mathfrak{p} = \langle x_1, \dotsc, x_t \rangle$. Then the blow up of $S$ along $\mathfrak{p}$, say $\Bl_{\mathfrak{p}}(S)$, can be covered by $t$ affine charts of the form $\Spec(S_i)$, with $S_i = S \left [ \frac{x_1}{x_i}, \dotsc, \frac{x_t}{x_i} \right ]$.
%\[
%	S_i = S \left [ \frac{x_1}{x_i}, \dotsc, \frac{x_t}{x_i} \right ].
%\]%
In addition, by Remark~\ref{rmk:cartas_blowup}, $\Bl_P(B)$ and $\Bl_{P'}(B')$ can be covered by affine charts of the form $\Spec(B_i)$ and $\Spec(B'_i)$ respectively, where $B_i$ contains and is finite over $S_i$, and $B'_i = B_i \left [ \frac{\theta_1}{x_i}, \dotsc, \frac{\theta_m}{x_i} \right ]$.
In order to check that the conditions of the lemma are preserved, we shall show that, if $F_s(B_i) \neq \emptyset$, then $\frac{\theta_1}{x_i}, \dotsc, \frac{\theta_m}{x_i} \in B'_i[W]$ are integral over $\G_{B_i}$ (the intrinsic Rees algebra attached to $F_s(B_i)$).

Let $\mathcal{H}_1$ denote the transform of $\mathcal{H}$ over the chart $\Spec(S_i) \subset \Bl_{\mathfrak{p}}(S)$, and suppose that $\mathcal{H} = \bigoplus_l J_l W^l \subset S[W]$. By Lemma~\ref{lm:elim-GB}, there is a finite inclusion $\mathcal{H} \subset \G_B$. Since $\theta_1 W, \dotsc, \theta_m W \in B'[W]$ are integral over $\G_B$, they are also integral over $\mathcal{H}$. Hence each $\theta_j$ satisfies an equation of integral dependence of the form $\theta_j^N + a_1 \theta_j^{N-1} + \dotsb + a_N = 0$
with $a_l \in J_l$. That is, with $a_l W^l \in \mathcal{H}$. Then one readily checks that $\frac{a_1}{x_i} W, \dotsc, \frac{a_N}{x_i^N} W^N \in \mathcal{H}_1$. Hence $\frac{\theta_j}{x_i} \in B'_i$ satisfies an equation of integral dependence of the form
\[
	\left( \frac{\theta_j}{x_i} \right)^N
		+ \frac{a_1}{x_i} \left( \frac{\theta_j}{x_i} \right)^{N-1}
		+ \dotsb
		+ \frac{a_N}{x_i^N}
	= 0,
\]%
which implies that $\frac{\theta_j}{x_i} W \in B'_i[W]$ is integral over $\mathcal{H}_1$. In addition, if $F_s(B_i) \neq \emptyset$, then $\mathcal{H}_1 \subset \G_{B_i}$ by Lemma~\ref{lm:elim-GB-blowup}. This proves that $\frac{\theta_j}{x_i} W \in B'_i[W]$ is integral over $\G_{B_i}$, and hence the conditions of the lemma are locally preserved.
\end{proof}

%%%%%%%%%%%%%%%%%%%%%%%%%%%%%%%%%%%%%%%%%%%%%%%%%%%%%%%%%%%%
%%%%%%%%%%%%%%%%%%%%%%%%%%%%%%%%%%%%%%%%%%%%%%%%%%%%%%%%%%%%
%%%
%%%  BIBLIOGRAPHY
%%%
%%%%%%%%%%%%%%%%%%%%%%%%%%%%%%%%%%%%%%%%%%%%%%%%%%%%%%%%%%%%
%%%%%%%%%%%%%%%%%%%%%%%%%%%%%%%%%%%%%%%%%%%%%%%%%%%%%%%%%%%%

%%% Una vez terminada la revision, se pueden comentar estas lineas
%\let\OldBibitem\bibitem
%\let\bibitemNO\bibitem
%\renewcommand{\bibitemNO}[1]{\color{green}\item[*]}
%\renewcommand{\bibitem}[1]{\color{black}\OldBibitem{#1}}

\end{document}